\documentclass[11pt]{amsart}
\usepackage{pb-diagram}
\usepackage{amscd}
\usepackage[mathscr]{eucal}
\usepackage{graphicx,amsmath,amssymb}
\usepackage[all]{xy}
\usepackage{multicol}

\newtheorem{theorem}{Theorem}[section]

\newtheorem{lemma}[theorem]{Lemma}

\newtheorem{cor}[theorem]{Corollary}
\newtheorem{conjecture}[theorem]{Conjecture}

\numberwithin{equation}{section}

\newtheorem{example}[theorem]{Example}
\newtheorem{definition}[theorem]{Definition}

\begin{document}

\title{Symplectic Classes on Elliptic Surfaces with positive Euler Number}
\author{Josef G. Dorfmeister}

\address{Department of Mathematics\\ North Dakota State University\\ Fargo, ND 58102}
\email{josef.dorfmeister@ndsu.edu}
\author{Tian-Jun Li}

\address{School  of Mathematics\\  University of Minnesota\\ Minneapolis, MN 55455}
\email{tjli@math.umn.edu}

\begin{abstract}
A key question for $4$-manifolds $M$ admitting symplectic structures is to determine which cohomology classes $\alpha\in H^2(M,\mathbb R)$ admit a symplectic representative.  The collection of all such classes, the symplectic cone $\mathcal C_M$, is a basic smooth invariant of $M$.  This paper describes the symplectic cone for elliptic surfaces with positive Euler number.
\end{abstract}

\maketitle

\tableofcontents

\section{Introduction}

Let $M$ be a smooth oriented $4$-manifold admitting symplectic structures.  The symplectic cone $\mathcal C_M\subset H^2(M,\mathbb R)$ is the collection of all classes $\alpha$ represented by an orientation compatible symplectic form $\omega$.  This cone has been determined in a number of cases, see \cite{H} for an overview.

The compatibility condition ensures that $\mathcal C_M\subset \mathcal P_M$, the set of classes with positive square.  A further restriction arises from Seiberg-Witten basic classes \cite{T}.  Both exceptional classes and symplectic canonical classes give rise to SW basic classes.  More precisely, let $\mathcal E$ denote the set of classes represented by smoothly embedded spheres of self-intersection $-1$.  Then it follows that for any $E\in\mathcal E$ and $\alpha\in\mathcal C_M$, $\alpha\cdot E\ne 0$.  Similarly, if $K=-c_1(M,\omega)$, then if $K\ne 0$, $K\cdot \alpha\ne 0$.

If $M$ is an elliptic surface, it is known that in certain cases the only constraints on a class to lie in the symplectic cone $\mathcal C_M$ are given by these three:  For those manifolds with $b^+(M)=1$, the results can be found in \cite{TJLL} and \cite{DL}.  For relatively minimal $T^2$-bundles over surfaces, the results can be found in \cite{Ge}, \cite{FV1}, \cite{FV2}, \cite{KN}, \cite{Wa}.   For relatively minimal $K3$-surfaces, the result is in \cite{TJL1}.  Non-minimal results are known for $T^4$, see \cite{EV} and \cite{LMS} (See Section \ref{s:t4non} for an alternative proof).   

This note extends the result to a large class of elliptic surfaces, the following is the main result, see Section \ref{s:cone}:

\begin{theorem}\label{t:main}
Let $M$ be an elliptic surface with $\chi(M)>0$, $p_g\geq 1$  and of Kodaira dimension 1.  Let $F_g$ be  a generic fiber with $F=[F_g]\in H_2(M,\mathbb Z)$.  
 Then 
\[
\mathcal C_{M}=\{\alpha\in\mathcal P_M\;|\; \alpha\cdot F\ne 0,\;\;\alpha\cdot E\ne 0\;\forall E\in\mathcal E\}.
\]

\end{theorem}

The same techniques make it possible to also prove a result in the non-minimal case for $K3$-surfaces (see also Lemma \ref{l:k3}).

\begin{lemma} Let $M=E(2)$.  Then 
\[
\mathcal C_{M\#l\overline{CP^2}}=\{\alpha\in\mathcal P_{M\#l\overline{CP^2}}\;|\;\alpha\cdot E_i\ne 0\;\forall i\in\{1,..,l\}\}
\]
\end{lemma}

For $V\subset M$ an oriented smooth submanifold, the relative symplectic cone $\mathcal C^V_M\subset \mathcal C_M$ consists of all classes such that a symplectic representative $\omega$ restricts to an orientation compatible symplectic form on $V$.  It follows, that for $\alpha$ to lie in the relative symplectic cone, $\alpha\cdot [V]>0$ must hold.  Hence $\mathcal C^V_M$ is always contained in the cone $\mathcal P_M^{[V]}\subset\mathcal P_M$ of classes which evaluate positively on $[V]$.  

It is an interesting question to consider how large the inclusion $\mathcal C_M\subset \mathcal P_M$ or $\mathcal C_M^V\subset \mathcal P_M^{[V]}$ is.  This is related to the conjecture below.

If $M$ underlies a minimal K\"ahler surface, then all symplectic forms have the same canonical class up to sign (\cite{FM}, \cite{W}).  Denote this class $K$.   If $b^+>1$, then Taubes \cite{T} has shown that the Poincar\'e dual to the canonical class $K$ is represented
by an embedded, symplectic curve.  In particular, this implies that for any symplectic class $\alpha$, $\alpha\cdot K\ne 0$.   Thus it follows that
\[
\mathcal C_M\subset \mathcal P_M^K\cup\mathcal P_M^{-K}.
\]
This leads to the following conjectures:
\begin{conjecture}\label{co:1}(\cite{TJL1}, Question 4.9)  If $M$ underlies a minimal K\"ahler surface with $b^+>1$, then
\[
\mathcal P_M^K\cup\mathcal P_M^{-K} \subset \mathcal C_M.
\] 
\end{conjecture}
This then implies that every class $\alpha$ of positive square with $\alpha\cdot K\ne 0$ is represented by a symplectic form.  A weaker version was stated by Hamilton:
\begin{conjecture}(\cite{H}, Conjecture 2)  Let $\overline{\mathcal C}_M$ be the closure of the symplectic cone in $H^2(M,\mathbb R)$. Then  
\[
\mathcal P_M^K\cup\mathcal P_M^{-K} \subset \overline{\mathcal C}_M.
\] 
\end{conjecture}
This would imply that the symplectic cone is dense in $\mathcal P_M^K\cup\mathcal P_M^{-K}$.

To determine the symplectic cones in Theorem \ref{t:main}, the relative symplectic cones of elliptic surfaces, relative to the generic fiber $F_g$, are determined (Theorems \ref{t:full} and \ref{t:mult}).

\begin{theorem}\label{t:relmain}
Let $M$ be an elliptic surface with $\chi(M)>0$ and $b^+>1$.  Let $F_g$ be an oriented
 generic fiber such that $\mathcal C^{F_g}_M\ne \emptyset$.  Let 
\[
\mathcal K(F_g)=\{K\in\mathcal K\;|\;K\cdot [F_g]=0\}
\]
be the set of symplectic canonical classes of $M$ which evaluate to 0 on $[F_g]=F$ and for $K\in\mathcal K$ denote
\[
\mathcal E_K=\{E\in\mathcal E\;|\;K\cdot E=-1\}.
\]  
Then
\[
\bigsqcup_{K\in\mathcal K(F_g)}\mathcal C_{M,K}^F=\mathcal C_M^{F_g}
\]
where 
\[
\mathcal C_{M,K}^F=\{\alpha\in\mathcal P^F_M\;|\; \alpha\cdot E>0\;\forall E\in\mathcal E_K\}.
\]
\end{theorem}

This result implies the following for relatively minimal elliptic surfaces with $b^+>1$ and positive Euler characteristic which admit symplectic structures:

\begin{cor}Let $M$ be a relatively minimal elliptic surface with $b^+(M)>1$ and $\chi(M)>0$.  Assume that  $\mathcal C_M\ne \emptyset$.  Then
\[
\mathcal P_M^F\cup \mathcal P_M^{-F}=\mathcal C_M.
\]
In particular, Conj \ref{co:1} holds if $M$ underlies a minimal K\"ahler surface with $b^+>1$, $\kappa(M)=1$ and positive Euler number.
\end{cor}

%
%
%
%
%
%
%
%
%
%
%
%

The proofs of Theorems \ref{t:main} and \ref{t:relmain} presented in this note break into two key parts:\begin{enumerate}
\item On the underlying smooth manifold, diffeomorphisms are used to control certain coefficients of classes lying in $\mathcal P_M^{[V]}$.  For $T^4$, such diffeomorphisms are explicitly constructed and their action on $H^2(M,\mathbb R)$ studied.  In the $\chi>0$ case, while the explicit diffeomorphisms are rather hard to come by, the structure of the geometric automorphism group $O$ (see Def \ref{d:aut}), this is the image of $Diff^+(M)$ in $H^2(M)$ modulo torsion, is rather well understood by the work of \cite{FM2}, \cite{Lo} and \cite{H2}.

The key results obtained from these automorphisms is the ability to reduce certain coefficients below any threshold to obtain a sum balanced class.  In other words, it becomes possible to concentrate the volume of a class $\alpha$ in certain terms and, when $M$ is written as a fiber sum, in one or the other summand as needed.  See for example Lemma \ref{l:t4} or Theorem \ref{t:mine1} for examples of this behavior.

These arguments are purely topological, they make no use of any symplectic arguments and also apply to elliptic surfaces with multiple fibers.  They are the content of Section \ref{s:3}.

\item Once a class has been made into a sum balanced class with respect to a splitting of $M$ as $X\#_{F_g}Y$ (see Def. \ref{d:sbal}), the class is split into three parts:  two parts lying wholly in $X$ or $Y$ and a rim torus component.  Using results in \cite{Go2} and \cite{H} and an inflation argument, it is then possible to show that a sum balanced class lies in the relative symplectic cone $\mathcal C_M^{F_g}$ if the corresponding cones in $X$ and $Y$ are understood.  This is the content of Section \ref{s:cp}.

\end{enumerate}

A short remark on notation:  $H_2(M,\mathbb R)$ and $H^2(M,\mathbb R)$ will rarely be distinguished.  In particular, automorphisms of $H_2$ and $H^2$ will not be distinguished.  
A generic fiber of the elliptic surface $M$ will be denoted by $F_g$, it's class by $F$.

{\bf Acknowledgements}  We would like to thank Bob Gompf for his interest and comments on our work as well as a suggestion for future work.  Further, we would like to thank Mark J. D. Hamilton for his careful reading of this manuscript and providing valuable feedback.

\section{Elliptic Surfaces}

Let $M$ denote an elliptic surface.  That is, $M$ is a complex surface admitting a holomorphic map to a complex curve $C$ of genus $g$ such that the generic fiber is a smooth elliptic curve.  $M$ is relatively minimal if no fiber contains an exceptional curve. 

Elliptic surfaces have been smoothly classified, this makes use of the fiber connected sum and multiple fibers.

\subsection{Singular Fibers}
An elliptic surface $M$ may have multiple and singular fibers, see \cite{BPV} for a precise classification. In the smooth category, it suffices to consider fishtail fibers and smooth multiple fibers \cite{GS}.

For any point in $x\in C\backslash\{x_1,...,x_t\}$, the fiber $F_x$ is smooth and diffeomorphic to $T^2$, this is the generic fiber $F_g$.  The fibers above $F_{x_i}$ may be of one of the following:\begin{enumerate}

\item A fishtail fiber (in Kodaira's notation $I_1$) is a sphere with a single transverse self-intersection point, locally modeled by $z_1z_2=0$.  Topologically, such a fiber can arise by collapsing a representative of one of the generators of $\pi_1(T^2)$ to a point.

%

\item A smooth multiple fiber is any fiber $F_{x_i}$ with $F=m[F_{x_i}]\in H_2(M,\mathbb Z)$, $m\ge 2$.  Kodaira \cite{Kod} described a procedure, called the logarithmic transform, to generate such fibers in the complex category.  Smoothly, this is a torus surgery, whereby a neighborhood $N$ of $F_g$ is removed and $T^2\times D^2$ is glued back in via a diffeomorphism $\phi:T^2\times \partial D^2\rightarrow \partial (M\backslash N)$, see\cite{GS}.  Denote the logarithmic transform of $M$ of multiplicity $p$ by $M(p)$.   Note that in the presence of multiple fibers, the generic fiber $F$ is no longer primitive.
\end{enumerate}

\subsection{Fiber Connected Sum} Let $M=X_1\#_{F_g}X_2$ be an elliptic surface obtained as the fiber sum of elliptic surfaces $X_i$ along a generic torus fiber $F_g$ by removing neighborhoods of $F_g$ in $X_1$ and $X_2$ and gluing along the boundary by an orientation reversing diffeomorphism.  The diffeomorphism will generally be implicit in the notation.

Any class $\alpha\in H_2(X_1\#_{F_g}X_2,\mathbb R)$ decomposes as follows (\cite{DL}, \cite{H1}):

\begin{equation}\label{e:dec}
\alpha=\alpha_{X_1}+\alpha_{X_2}+\alpha_F+\alpha_{RT}.
\end{equation}

In this decomposition, $\alpha_F$ consists of the class $F=[F_g]$ of the submanifold along which the sum is performed and a class $\Gamma$ composed out of elements of the homology of both $X_i$, a type of "section", which intersects $F_g$ non-trivially.  The class $\alpha_{RT}$ is composed of two pairs $(\mathcal R_i, S_i)$ which are rim tori and dual vanishing classes generated in the fiber sum, but which do not exist in either $X_i$.  Depending on the decomposition of $M$, this class may exist or be empty.  Finally, the classes $\alpha_{X_i}$ contain all classes of $X_i$ which are supported away from a neighborhood of the submanifold $F_g$ and intersect $F$ and $\Gamma$ trivially (hence especially does not include the fiber class).

Hence for $M$ an elliptic surface, this decomposition satisfies
\[
\alpha_{X_i}\cdot\alpha_F=\alpha_{X_i}\cdot \alpha_{RT}=\alpha_F\cdot \alpha_{RT}=0,
\]
\[
F\cdot\Gamma\ge 1,\mbox{  }\mathcal R_i\cdot S_j=\delta_{ij},
\]
and
\[
F^2=\mathcal R_i\cdot\mathcal R_j=0.
\]

In the presence of multiple fibers, the generic fiber class $F$ is no longer primitive, let $F=\tau f$, $\tau\in\mathbb Z$ and $f$ a primitive class.  It is possible to choose $\Gamma$ such that $f\cdot \Gamma=1$.  In the absence of multiple fibers, the class $\Gamma$ can be chosen to be the class of a smooth section.

\subsection{Examples of Elliptic Surfaces}  Using the fiber sum, elliptic surfaces can be constructed from a few basic surfaces.  Theorem \ref{t:class} will state that if $\chi(M)>0$, then this describes all elliptic surfaces.  This also introduces notation that will be subsequently used.

Let $L(p_1,...,p_k)$ be the $T^2$-bundle $S^2\times T^2$ with multiple torus fibers of multiplicities $(p_1,..,p_k)$, $k\ge 1$, $T^2\times \Sigma_g$ be the trivial $T^2$-bundle over a closed surface of genus $g$ and $E(1)=\mathbb CP^2\#9\overline{\mathbb CP^2}$.  Inductively define 
\[
E(n)=E(n-1)\#_{F_g}E(1),
\]
\[
E(n,g)=E(n)\#_{F_g}(T^2\times \Sigma_g)
\]
and
\[
E(n,g,p_1,..,p_k)=E(n,g)\#_{F_g}L(p_1,..,p_k)
\]
This defines a relatively minimal elliptic surface over a curve of genus $g$ with multiple fibers of multiplicities $(p_1,...,p_k)$ and with $\chi(E(n,g,p_1,...,p_k))=12n>0$.

\subsection{Smooth Classification}

Every relatively minimal elliptic surface $M$ arises from $E(1)=\mathbb CP^2\#9\overline{\mathbb CP^2}$ and $T^2$-bundles over $T^2$ using the fiber sum along a generic smooth fiber $F_g$  and logarithmic transforms.  In the case that $\chi(M)=0$, then the only singular fibers are multiple fibers.  If $\chi(M)>0$, then $M$ must contain an $E(1)$-summand.  The following Theorem gives a full classification of relatively minimal elliptic surfaces up to diffeomorphism.

\begin{theorem}\label{t:class}Let $M$ be a relatively minimal elliptic surface.\begin{enumerate}
\item (\cite{Ue},\cite{Ma}, see also \cite{FM2}, \cite{Ue2}, \cite{GS} ) Assume that $\chi(M)=0$. Then $M$ is obtained from a torus bundle over an orientable surface $\Sigma_g$ ($g\ge 0$) by logarithmic transforms.  The diffeomorphism type is determined by the fundamental group of $M$.  
\item (Thm. 8.3.12, \cite{GS}) Assume that $\chi(M)\ne 0$.  Then $M$ is diffeomorphic to $E(n,g,p_1,..,p_k)$  for exactly one choice of $(n,g,p_1,..,p_k)$, where $n\ge 1$, $g,k\ge 0$ and $2\le p_i$.  If $(n,g)=(1,0)$, then $k\ne 1$. 
\end{enumerate}

\end{theorem}

Note that $E(1)$ is diffeomorphic to $E(1,0,p)$, hence the final condition in the theorem.  The diffeomorphism sends the multiple fiber class $F_p$ of $E(1,0,p)$ to the fiber class $F$ of $E(1)$.  

For a relatively minimal elliptic surface $M$ with $\chi(M)=12n$ over $\Sigma_g$ and  with given fiber class $F$, the canonical divisor $\mathcal K_{min}$ is given by (\cite{FM2}, \cite{BPV})
\begin{equation}\label{e:k}
\mathcal K_{min}= (2g-2+n+k)F-\sum_{i=1}^kF_{p_i}
\end{equation}
where $F_{p_i}$ are the classes of the multiple fibers with $F=p_iF_{p_i}$.  Note that for $2g-2+n+k=0$ and $k=0$, $\mathcal K_{min}= 0$.  Denote by $K_{min}$ the corresponding canonical class. 

\subsection{Kodaira Dimension}  Kodaira dimension $\kappa(M)$ is defined on the minimal model for complex \cite{BPV} and symplectic manifolds (\cite{L1}, \cite{LB1}, \cite{MS}) and, when both are defined, they coincide \cite{DZ}.  For $4$-manifolds, it takes values in $\{-\infty, 0,1,2\}$ and elliptic surfaces satisfy $\kappa(M)\le 1$.  If $M$ is not minimal, then its Kodaira dimension is that of its minimal model.

Assume that $M$ is relatively minimal.  \begin{enumerate}
\item $\kappa(M)=-\infty$:  Then $M$ is diffeomorphic to $E(1)$, a Hopf surface or an $S^2$-bundle over $T^2$ with at most three multiple fibers (see I.3.23, \cite{FM2}).
\item $\kappa(M)=0$: Then $M$ is $E(2)$, an Enriques surface ($\simeq E(1,0,2,2)$), a Kodaira surface or a $T^2$-bundle over $T^2$.  In particular, $K_{min}= 0$ or $2K_{min}=0$.
\item All other elliptic surfaces have $\kappa(M)=1$.
\end{enumerate}

Denote 
\[
\delta=2g-2+n+k-\sum\frac{1}{p_i},
\]
thus $K_{min}=\delta F\in H^2(M,\mathbb R)$.   Furthermore, note that for $\alpha\in H^2(M,\mathbb R)$ and $K_{min}\ne 0$ or torsion,
\[
\alpha\cdot K_{min}\ne 0\Leftrightarrow \alpha\cdot F\ne 0.
\] 

\begin{lemma}\label{l:sign}(V.12.5, \cite{BPV}) Let $M$ be a relatively minimal elliptic surface.  Then
\[
\kappa(M)=\left\{\begin{matrix}-\infty\\0\\1\end{matrix}\right\}\Leftrightarrow \delta\left\{\begin{matrix}<\\=\\>\end{matrix}\right\}0.
\]
\end{lemma}

An immediate consequence in the case $\kappa(M)=1$ is for $\alpha\in H^2(M,\mathbb R)$,
\begin{equation}\label{e:sasi}
\alpha\cdot K_{min}> 0\Leftrightarrow \alpha\cdot F> 0.
\end{equation}

\subsection{Positive Euler Characteristic}   Theorem \ref{t:class} shows that every relatively minimal $M$ with positive Euler characteristic is diffeomorphic to some $E(n,g,p_1,...,p_k)$.  This manifold can be split as
\[
E(n,g,p_1,...,p_k)=E(1)\#_{F_g}E(1)\#_{F_g}....\#_{F_g}E(1)\#_{F_g}(T^2\times \Sigma_g)\#_{F_g}L(p_1,..,p_k)
\]
which leads to the intersection form 
\[
E_8\oplus P_1\oplus E_8\oplus P_2\oplus E_8\oplus...\oplus P_{n-1}\oplus E_8\oplus ({f},\Gamma)\oplus \left(b^+(T^2\times \Sigma_g)-1\right) H
\]

This means the following:
\begin{enumerate}

\item The $E_8$ intersection component is given by the matrix
\[
E_8=\left(\begin{matrix}
-2 & 0 & 0 & 1 & 0 & 0 & 0 & 0 \\
0 & -2 & 1 & 0 & 0 & 0 & 0 & 0 \\
0 & 1 & -2 & 1 & 0 & 0 & 0 & 0 \\
1 & 0 & 1 & -2 & 1 & 0 & 0 & 0 \\
0 & 0 & 0 & 1 & -2 & 1 & 0 & 0 \\
0 & 0 & 0 & 0 & 1 & -2 & 1 & 0 \\
0 & 0 & 0 & 0 & 0 & 1 & -2 & 1 \\
0 & 0 & 0 & 0 & 0 & 0 & 1 & -2
\end{matrix}\right).
\]
This arises from the intersection form on $E(1)$ as $\langle1\rangle\oplus 9\langle -1\rangle=E_8\oplus H'$ where \[
H'=\left(\begin{matrix}
0&1\\1&-1
\end{matrix}\right).
\] 
Hence there are as many $E_8$-terms as there are $E(1)$ summands.
\item Each $P_i$ consists of the two rim pairs $(\mathcal R_i,S_i)$, $\mathcal R_i$ represented by a rim torus, $S_i$ representable by an embedded sphere of self-intersection -2.  It will be convenient to change this pair to $(\mathcal R_i, T_i=\mathcal R_i+S_i)$.   This new pair contributes a copy of $H$ to the intersection form, i.e. each $P_i=H\oplus H=2H$.  Note that if $\alpha_{i,j}=e_j\mathcal R_j+d_jT_j$, then 
\[
\alpha_{i,j}^2=2e_jd_j.
\]
Hence the areas of $\mathcal R_j$ and $T_j$ have the same sign if and only if $\alpha_{i,j}^2>0$.

Rim pairs only arise when the summands on either side contain a $E(m)$ component.

\item The term $({f},\Gamma)$ corresponds to a generic fiber $F=\tau f$, $f$ primitive,  and a "section" class $\Gamma$, this pair has intersection matrix
\[
\left(\begin{matrix}
0&1\\1&\Gamma^2
\end{matrix}\right).
\]
If $M$ has no multiple fibers, then $\Gamma$ can be represented by a smooth section of the fibration and $\Gamma^2=-n$.  

In the following, whenever possible, the generic fiber class $F$ will be used.  In particular, 
\[
\alpha_F=cF+g\Gamma
\]
 and hence $\alpha_F^2=2cg\tau +g^2\Gamma^2$.
 
\item The final term arises from the summation with $T^2\times \Sigma_g$ and contributes $2gH$ to the intersection form.
\end{enumerate}

Thus the intersection form of $E(n,g,p_1,...,p_k)$ can be written more succinctly as
\[
=nE_8\oplus \left[2(n-1)+2g\right]H\oplus \left(\begin{matrix} 0&1\\1 &\Gamma^2\end{matrix}\right).
\]
This decomposition is pairwise orthogonal and given a class $\alpha\in H_2(M)$, we can write
\begin{equation}\label{e:n}
\alpha=\underbrace{\sum_{i=1}^n\alpha_{8,i}}_{E_8\mbox{ terms}}+\alpha_F+\sum_{i=1}^{n-1}\alpha_{P_i}+\alpha_{T^2\times \Sigma_g}=\sum_{i=1}^n\alpha_{8,i}+\sum_{i=1}^{2(n-1)+2g}\alpha_{H,i}+\alpha_F.
\end{equation}
In this notation, each $\alpha_{H,i}=aA+bB$ represents one $H$-term, i.e. the intersection pattern for $A$ and $B$ is given by $H$. 

Once a choice of $M=X_1\#_{F_g}X_2$ has been made, then one of two situations can occur:  If both $X_1$ and $X_2$ have a $E(m)$-type summand, then a certain $P_j$ arises as the rim component of this sum.  The remaining $P_i$ terms lie in either $X_1$ or $X_2$, providing each a $2H$-contribution to the intersection form.  Then \eqref{e:n} becomes
\begin{equation}\label{e:en}
\alpha=\sum_{i=1}^n\alpha_{8,i}+\sum_{i=1}^{2(n-1)+2g-2}\alpha_{H,i}+\underbrace{e_{j,1}\mathcal R_{j,1}+d_{j,1}T_{j,1}+e_{j,2}\mathcal R_{j,2}+d_{j,2}T_{j,2}}_{\alpha_{RT}}+\alpha_F.
\end{equation}
and the decomposition \eqref{e:dec} has
\[
\alpha_{X_k}=\sum_{i=1}^{n_k}\alpha_{8,i}+\sum_{i=1}^{m_k}\alpha_{H,i}
\]
with $n_k\ge 1$, $m_k\ge 0$, $n_1+n_2=n$ and $m_1+m_2=2(n-1)+2g$.

If all the $E(m)$-type components lie in one $X_i$, then the sum involves no rim pairs and then \eqref{e:dec} becomes
\begin{equation}
\alpha=\alpha_{X_1}+\alpha_{X_2}+\alpha_F.
\end{equation}
In this case,
\[
\alpha_{X_1}=\sum_{i=1}^{n}\alpha_{8,i}+\sum_{i=1}^{m_1}\alpha_{H,i}
\]
and
\[
\alpha_{X_2}=\sum_{i=1}^{m_2}\alpha_{H,i}
\] 
with $m_1+m_2=2(n-1)+2g$.

It will be convenient to write the class $\alpha$, or parts of it under consideration, in vector notation:  For example, 
\[
\alpha_{RT}=e_{j,1}\mathcal R_{j,1}+d_{j,1}T_{j,1}+e_{j,2}\mathcal R_{j,2}+d_{j,2}T_{j,2}=(e_{j,1},d_{j,1},e_{j,2},d_{j,2}).
\]
Further, as the precise choice of $P_j$ will not be relevant, this will further be shortened to
\[
\alpha_{RT}=(e_{1},d_{1},e_{2},d_{2}).
\]

The aim of this note is to determine which classes $\alpha$ can be represented by symplectic forms on $M$.  The underlying tactic is to use the decomposition of $M$ as a fiber sum $X\#_FY$ to answer this question by relating $\alpha$ to symplectic classes $\alpha_X$ and $\alpha_Y$ on $X$ and $Y$.  As the decomposition in \eqref{e:en} shows, an additional issue is the presence of rim components.  In the following, these issues will be first addressed at a homological level (see below and Section \ref{s:3}) and then at a geometric level (see Def. \ref{d:sbal} and Def. \ref{d:pfc} in Section \ref{s:4}).

Initially, there are three straightforward aspects that need to be considered:  First, a very basic criterion for $\alpha$ to be symplectically representable with respect to the fiber sum is that $\alpha^2>0$ and $\alpha\cdot F>0$.  This motivates the following definition.

\begin{definition} The {\it positive cone} is 
\[
\mathcal P_M=\{\alpha\in H^2(M,\mathbb R)|\;\alpha^2>0\}
\]
and for a nonzero class $A\in H_2(M,\mathbb Z)$, the {\it relative positive cone} is
\[
\mathcal P_M^A=\{\alpha\in H^2(M,\mathbb R)|\;\alpha^2>0,\;\alpha\cdot A>0\}
\]
and $\mathcal P_M^0=\mathcal P_M$.
\end{definition}

In relation to the conjectures of the introduction, note that if $\kappa(M)=1$, then \ref{e:sasi} shows that the relative positive cones for $K_{min}$ and for $F$ are identical.

Secondly, if $\alpha$ has a non-vanishing $\alpha_{RT}$ term, then $\alpha$ will not directly be described by only terms in $X$ and $Y$.  In this case $\alpha_{RT}$ will need to have a specific form, this leads to the concept of a balanced class.

\begin{definition} \label{d:bal}Let $M=X_1\#_{F_g}X_2$ be an elliptic surface and $\alpha\in\mathcal P^F_M$.  Then $\alpha$ is balanced with respect to $(X_1,X_2)$ if either 
\begin{enumerate}
\item the sum has no rim pairs or 
\item if there are the rim pairs generated in the sum of $X_1$ and $X_2$, then the decomposition of $\alpha$ given by \ref{e:en} satisfies:
\begin{enumerate}
\item $e_i\cdot d_i> 0$ or $e_i=d_i= 0$ for $i\in\{1,2\}$ and
\item $\alpha^2-2e_1\cdot d_1-2e_2\cdot d_2>0$.
\end{enumerate}
\end{enumerate}
\end{definition}

Finally, the criterion $\alpha^2>0$ will need to hold for a splitting of 
\[
\alpha-\alpha_{RT}=\alpha_{X_1}+\alpha_{X_2}+\alpha_{F_{X_1}}+\alpha_{F_{X_2}},
\]
 i.e. $\alpha_{X_i}^2+\alpha_{F_{X_i}}^2>0$ for $i\in\{1,2\}$.  This homological sum balanced criterion (see Def \ref{d:sbal}) can be attained by a correct choice of splitting $\alpha_F=\alpha_{F_{X_1}}+\alpha_{F_{X_2}}$.
 
The geometric arguments of Section \ref{s:4} build on these basic homological properties to ensure that a given class can be represented symplectically.  In particular, sum balanced ensures that a class is split into two symplectic classes and partially fibration compatible at $F_g$ ensures that the geometric argument of Theorem \ref{t:sym} can be properly executed.

\section{Diffeomorphism Groups and Homology Actions for Elliptic Surfaces with Positive Euler Number}\label{s:3}

A self-diffeomorphism of $M$ always induces an automorphism of $H_2(M,\mathbb Z)$, and by extension of $H_2(M,\mathbb R)$.  The converse is unfortunately not always the case.  In the following we describe automorphisms which are shown to cover a self-diffeomorphism of $M$.  The results in this section are valid for any elliptic surface, including those with multiple fibers.

\begin{definition}
Two classes $\alpha,\alpha'\in H_2(M,\mathbb R)$ are equivalent if there exists an automorphism in the image of $Diff^+(M)$ mapping one to the other.
\end{definition}

\subsection{Automorphism Groups of Elliptic Surfaces }  For relatively minimal elliptic surfaces with positive Euler number, the image of $Diff^+(M)$ in $Aut(H_2(M,\mathbb Z))$ (modulo torsion) is rather well understood.  

\begin{definition}\label{d:aut}(\cite{Lo}, \cite{FM2}, \cite{H2}) Let $M$ be an elliptic surface.\begin{enumerate}
\item Denote $\overline{H}_2(M)$ the second homology $H_2(M,\mathbb Z)$ modulo torsion.
\item Denote by $O$ the orthogonal group of automorphisms of $\overline{H}_2(M)$ which preserve the intersection form.
\item Let $k$ be a canonical divisor on $M$.  Denote by $O_k\subset O$ the automorphisms which fix $k$.  If $k=0$, then $O_k=O$.
\item For an element $\phi\in O$, define its spinor norm to be $\pm 1$ depending on whether $\phi$ preserves or reverses the orientation of a maximal positive definite subspace of $H_2(M,\mathbb R)$.  The spinor norm is a group homomorphism, i.e. the spinor norm of $\phi\circ \psi$ is the product of the respective spinor norms.
\item Denote by $O'\subset O$ the subgroup of elements of spinor norm 1.
\end{enumerate}
\end{definition}

\begin{theorem} \label{t:lo}\cite{Lo} Let $M$ be a relatively minimal elliptic surface with positive Euler number and $k$ the canonical class.  Then the image of $Diff^+(M)$ in $O$ is $O'$ if $M=E(2)$ and contains $O_k'$ otherwise. 
\end{theorem}

In \cite{Lo} a detailed construction of the image is undertaken.  It is shown, that the induced automorphisms are generated by reflections on spheres of self-intersection $-2$.  Reflection along a smoothly embedded sphere $S$ with $S^2=-2$ induces a diffeomorphism of $M$ (III.Prop 2.4, \cite{FM1}), on homology the map is given by
\[
r_{S}(B)=B-2\frac{B\cdot S}{S\cdot S}S=B+(B\cdot S)S.
\]  
Moreover, this diffeomorphism is identity outside any neighborhood of the sphere. A broad class of such spheres is given by the following Theorem.

\begin{theorem}\label{t:2}\cite{Lo} If $M$ is a relatively minimal elliptic surface with non-vanishing Euler number and $F_g$ a generic fiber, then every class in $\overline{H}_2(M\backslash F_g)$ of square $-2$ is represented by a sphere smoothly embedded in $M\backslash F_g$.  Moreover, every reflection on such a class is realized by a diffeomorphism which is identity on a neighborhood of $F_g$.  This  includes all automorphisms of spinor norm one.
\end{theorem}

The constructions below map a given class $\alpha\in\mathcal P^F_M$ using a finite sequence of automorphisms to a special class.  This will be done using either reflections on $-2$-spheres or maps of spinor norm one.  Thus, the underlying diffeomorphism of the composition avoids a neighborhood of a fiber, it is in this neighborhood that the blow-up locus is assumed to be located.

There are now two pathways to construct automorphisms in the image of $Diff^+(M)$:\begin{enumerate}
\item Identify a $-2$-sphere and reflect on it.  Theorem \ref{t:2} implies that any class of square $-2$ which is in
\[
nE_8\oplus \left[2(n-1)+2g\right]H
\] 
is represented by a smoothly embedded sphere.  

Note that if $n$ is even, then the intersection pairing given by $(F,\Gamma)$ is equivalent to $H$.
\item Construct a map that preserves the intersection form, has spinor norm 1 and, if needed, preserves the canonical class, i.e. which lies in $O_k'$.  If $M\ne E(2)$, then preserving the canonical class is equivalent to preserving the fiber $F$, see Lemma \ref{l:sign}.
\end{enumerate}

\subsection{Explicit Automorphisms of Elliptic Surfaces}  

Not every class $\alpha\in \mathcal P^{F}_M$ is balanced, the goal is to find automorphisms in $O_k'$ of $M$ (or $O'$ if $M=E(2)$) that map the  class $\alpha$ to an equivalent balanced class.  This section describes the automorphisms that will be used to achieve this.

The first three maps will often be used to re-organize classes representing an $H$ or $2H$-term in the intersection form.

\subsubsection{Reflection on $H$}

Let $(A,B)$ be a pair with intersection form given by $H$.  Then $B-A$ squares to $-2$.  The map on homology only affects $aA+bB$.  This map acts by 
\[
aA+bB\;\;\mapsto\;\; bA+aB.
\]

This map will be used on the following two pairs:\begin{enumerate}
\item A rim pair $(\mathcal R, T)$, the underlying vanishing sphere $S=T-\mathcal R$ is a smoothly embedded $-2$-sphere.
\item A pair of tori $(T_1, T_2)$ arising in $T^2\times \Sigma_g$ as a summand in an elliptic surface with positive Euler number.  These are disjoint from the fiber, hence by Theorem \ref{t:2}, $T_i-T_j$ is represented by a smoothly embedded $-2$-sphere in the complement of a generic fiber.
\end{enumerate}

\subsubsection{Q-map} (Lemma 2.5, \cite{H2})  Any map $Q$ which acts by $-id$ on two $H$-components and by identity on the remainder of the class has spinor norm 1.  As it leaves the fiber class unchanged, this map covers a self-diffeomorphism of $M$.

A map $Q$ which only acts by $-id$ on one $H$-component and identity otherwise has spinor norm -1.  This map can be used to adjust any map from spinor norm -1 to 1 at the cost of sign changes on $H$.

\subsubsection{Interchange Map}

Let $(A_1,B_1,A_2,B_2)$ generate a $2H$ term in the intersection form.  Define a map by
\[
\begin{array}{lcl}
A_1&\mapsto &A_2\\
B_1&\mapsto & B_2\\
 A_2&\mapsto& -A_1\\
B_2&\mapsto &-B_1
\end{array}
\]
 which otherwise acts by identity. This map lies in $O_{k}'$.  This can be seen as the naive interchange map $(a_1,b_1,a_2,b_2)\mapsto (a_2,b_2,a_1,b_1)$ has spinor norm one (\cite{H3}) and then apply a Q-map.  Alternatively, this map arises from repeated applications of Lemma\ref{l:rp} below and reflections on $-2$-spheres.   
 
The interchange map can be applied to any two pairs of tori arising in an elliptic surface with positive Euler number.

\subsubsection{Automorphism of the lattice $2H$}  

Let $(A_1,B_1,A_2,B_2)$ generate a $2H$ term in the intersection form.

\begin{lemma}\label{l:rp} Let $M$ be an elliptic surface and $i\in\mathbb Z$.  The automorphism of $H_2(M,\mathbb Z)$ defined by
\[
\begin{array}{lcl}
A_1&\mapsto &A_1-iA_2\\
B_1&\mapsto & B_1\\
 A_2&\mapsto& A_2\\
B_2&\mapsto &B_2+iB_1
\end{array}
\]
and otherwise acting by identity is induced by a self-diffeomorphism of $M$.
\end{lemma}

\begin{proof}
This is proven identically to Lemma 5.1, \cite{H}. 
\end{proof}

The action of this map is
\[
(a_1,b_1,a_2,b_2)\;\mapsto\;(a_1,b_1+ib_2,a_2-ia_1,b_2).
\]
This map shifts volume from one pair to the other while also changing the area on one term in each pair.  

This can be applied to any two pairs of (rim) tori.

If $\Gamma^2=2m$ is even, then the pair $(F, \Gamma-mF=W)$ has $H$ as its intersection form.  Then the map in Lemma \ref{l:rp} can be applied, however note that this will not preserve the fiber class.

\begin{lemma}\label{l:e2}  Let $M$ be an elliptic surface with a given fibration having $F_g$ as a generic fiber and $i\in\mathbb Z$.  Let $\alpha\in H_2(M,\mathbb Z)$ be of the form 
\[
\alpha=\alpha_0+aA+bB+wF+gW.
\]  
The automorphism of $H_2(M,\mathbb Z)$ defined by Lemma \ref{l:rp} sending $\alpha$ to the class  
\[
\tilde \alpha=\alpha_0+(a+iw)A+bB+w\tilde F+(g-ib)W
\]
has spinor norm one and changes the fiber class.
\end{lemma}

\begin{proof}  Lemma \ref{l:rp} provides a map of spinor norm one as follows:
\[
\begin{array}{lcl}
A_1&\mapsto &A_1-iA_2\\
B_1&\mapsto & B_1\\
 A_2&\mapsto& A_2\\
B_2&\mapsto &B_2+iB_1
\end{array}
\]
and otherwise acting by identity is induced by a self-diffeomorphism of $M$.  Previously, this map has been applied in the same basis of $H_2(M)$ in domain and codomain.  Now, view this map as a change of basis map.  The basis elements $(A_1,B_1,A_2,B_2)$ get mapped to a new basis 
\[
(\tilde A_1=A_1-iA_2,\tilde B_1=B_1,\tilde A_2=A_2,\tilde B_2=B_2+iB_1).
\]
In this new basis, $A_1=\tilde A_1+i\tilde A_2$ and $B_2=\tilde B_2-i\tilde B_1$.  This means
\[
a_1A_1+b_1B_1+a_2A_2+b_2B_2 = a_1\tilde A_1+(b_1-ib_2)\tilde B_1+(a_2+ia_1)\tilde A_2+b_2\tilde B_2.
\]

Applying this to $(A,B)$ and $(F,\Gamma)=(A_1,B_1)$ leads to the claim.
\end{proof}

Let $M$ be diffeomorphic to $E(2)$, then by Theorem \ref{t:lo}, automorphisms covering self-diffeomorphisms of $M$ no longer need to preserve the fiber class, but still need to have spinor norm one.  This result will be applied in that setting.

\subsubsection{H-Fiber Map} Let $(A,B)$ generate an $H$-term.  The following map is related to the $2H$-map as in Lemma \ref{l:rp}, but depending on the parity of $\Gamma^2$ may not arise from $2H$-terms.  Due to the restrictions arising from Theorem \ref{t:lo}, this is the only map involving the fiber $F$.

\begin{lemma}\label{l:fi} Let $M$ be an elliptic surface and $i\in\mathbb Z$.  The automorphism of $H_2(M,\mathbb Z)$ defined by
\[
\begin{array}{lcl}
f&\mapsto &f\\
\Gamma&\mapsto & \Gamma+iB\\
A&\mapsto& A-if\\
B&\mapsto &B
\end{array}
\]
 and otherwise acting by identity is induced by a self-diffeomorphism of $M$.

\end{lemma}

This map is a generalization of the map given in Lemma 5.1, \cite{H}, and the proof of this lemma is identical to the one given there.

The action of this map is as follows:
\[
(c,g,a,b)=c\tau f+g\Gamma+aA+bB\mapsto 
\]
\[
(c\tau-ia)f+g\Gamma+aA+(b+ig)B= \left(c-\frac{i}{\tau}a,g,a,b+ig\right).
\]

\subsubsection{Automorphisms of the lattice $E_8\oplus H$}
 Consider the lattice $E_8\oplus H$; let
\[
\alpha=\sum_{i=0}^7k_iD_i+aA+bB
\] 
be a point in this lattice.  The class $\pm D_i + A$ has self-intersection $-2$ and hence is represented by a smoothly embedded sphere.  This leads to the following lattice automorphisms:
\begin{enumerate}

\item[(1a)] Reflection along $D_i+A$, $i\in\{2,4,5,6\}$:  This automorphism is given by
\[\begin{array}{lcl}
D_{i-1}&\mapsto&D_{i-1}+D_i+A\\
 D_i&\mapsto&-D_i-2A\\
 D_{i+1}&\mapsto&D_i+D_{i+1}+A\\
 A&\mapsto&A\\
  B&\mapsto&  B+D_i+A
  \end{array}
\] 
and identity on the remainder.  It changes the coefficients 
\[\begin{array}{lcl}
k_i&\mapsto& k_{i-1}-k_i+k_{i+1}+b\\
a&\mapsto &a+k_{i-1}-2k_i+k_{i+1}+b
\end{array}
\] 
while leaving all others unchanged.

\item[(1b)] Reflection along $-D_i+A$, $i\in\{2,4,5,6\}$:  This automorphism changes the coefficients 
\[\begin{array}{lcl}
k_i&\mapsto& k_{i-1}-k_i+k_{i+1}-b\\
a&\mapsto &a-k_{i-1}+2k_i-k_{i+1}+b
\end{array}
\] 
while leaving all others unchanged.

\item[(1c)]  Combining the automorphisms in (1a) and (1b) by performing first one reflection and then the other produces two automorphisms which again only change the $k_i$ and $a$ coefficients
:\begin{enumerate}
\item $(-D_i+A)\circ (D_i+A)$:
\[\begin{array}{lcl}
k_i&\mapsto& k_i-2b\\
a&\mapsto &a+2k_{i-1}-4k_i+2k_{i+1}+4b
\end{array}
\] 
\item $(D_i+A)\circ (-D_i+A)$:
\[\begin{array}{lcl}
k_i&\mapsto& k_i+2b\\
a&\mapsto &a-2k_{i-1}+4k_i-2k_{i+1}+4b
\end{array}
\] 

\end{enumerate}

\item[(2a)] Reflection along $D_i+A$ for $i\in\{0,1,7\}$:  Consider the pairs $(i,j)\in\{(0,3), (1,2) (7,6)\}$.  This automorphism is given by
\[\begin{array}{lcl}
D_j&\mapsto&D_j+D_i+A\\
 D_i&\mapsto&-D_i-2A\\
 A&\mapsto&A\\
  B&\mapsto&  B+D_i+A
  \end{array}
\] 
and identity on the remainder.  It changes the coefficients 
\[\begin{array}{lcl}
k_i&\mapsto &k_j-k_i+b\\
a&\mapsto &a+k_j-2k_i+b
\end{array}
\] 
while leaving all others unchanged.

\item[(2b)] Reflection along $-D_i+A$ for $i\in\{0,1,7\}$:  Consider the pairs $(i,j)\in\{(0,3), (1,2) (7,6)\}$.  This automorphism changes the coefficients 
\[\begin{array}{lcl}
k_i&\mapsto &k_j-k_i-b\\
a&\mapsto &a-k_j+2k_i+b
\end{array}
\] 
while leaving all others unchanged.

\item[(2c)]  Combining the automorphisms in (2a) and (2b) by performing first one reflection and then the other produces two automorphisms which again only change the $k_i$ and $a$ coefficients
:\begin{enumerate}
\item $(-D_i+A)\circ (D_i+A)$:
\[\begin{array}{lcl}
k_i&\mapsto& k_i-2b\\
a&\mapsto &a+2k_{j}-4k_i+4b
\end{array}
\] 
\item $(D_i+A)\circ (-D_i+A)$:
\[\begin{array}{lcl}
k_i&\mapsto& k_i+2b\\
a&\mapsto &a-2k_{j}+4k_i+4b
\end{array}
\] 

\end{enumerate}

\item[(3a)] Reflection along $D_3+A$:  This automorphism is given by
\[\begin{array}{lcll}
D_j&\mapsto&D_j+D_3+A&j\in\{0,2,4\}\\
 D_3&\mapsto&-D_3-2A&\\
 A&\mapsto&A&\\
  B&\mapsto&  B+D_3+A&
  \end{array}
\] 
and identity on the remainder.  It changes the coefficients 
\[\begin{array}{lcl}
k_3&\mapsto &k_0+k_2-k_3+k_4+b\\
a&\mapsto &a+k_0+k_2-2k_3+k_4+b
\end{array}
\] 
while leaving all others unchanged.

\item[(3b)] Reflection along $-D_3+A$:  This automorphism changes the coefficients 
\[\begin{array}{lcl}
k_3&\mapsto &k_0+k_2-k_3+k_4-b\\
a&\mapsto &a-k_0-k_2+2k_3-k_4+b
\end{array}
\] 
while leaving all others unchanged.

\item[(3c)]  Combining the automorphisms in (3a) and (3b) by performing first one reflection and then the other produces two automorphisms which again only change the $k_i$ and $a$ coefficients
:\begin{enumerate}
\item $(-D_3+A)\circ (D_3+A)$:
\[\begin{array}{lcl}
k_3&\mapsto& k_3-2b\\
a&\mapsto &a+2k_{0}+2k_2-4k_3+2k_4+4b
\end{array}
\] 
\item $(D_3+A)\circ (-D_3+A)$:
\[\begin{array}{lcl}
k_3&\mapsto& k_3+2b\\
a&\mapsto &a-2k_{0}-2k_2+4k_3-2k_4+4b
\end{array}
\] 

\end{enumerate}

\end{enumerate}

The maps given in (1c), (2c) and (3c) can be applied repeatedly to change the $k_i$ coefficient by any even integer multiple of b.

\begin{lemma}\label{l:e8h}
Let $M$ be an elliptic surface which has an $E_8\oplus H$ component in the intersection form.  Let $r=(r_0,...,r_7)\in\mathbb Z^8$.  Then there exists an automorphism $A(r_0,...,r_7)$ of $\overline{H}_2(M,\mathbb Z)$, covering a self-diffeomorphism of $M$, which acts only on $E_8\oplus H$ and is identity on all other components.  The action on a point $\sum_{i=0}^7k_iD_i+aA+bB\in E_8\oplus H$ is given by
\[
k_i\mapsto k_i+2br_i
\]
and, writing $k=(k_0,...,k_7)^T$,  
\[
a\mapsto a+4b\sum r_i+2(r^T\cdot E_8\cdot k)
\]
 while $b$ is left unchanged.

\end{lemma}

This map will be used to change the volume of any $E_8$-component to be as close to 0 as possible while also ensuring that the individual coefficients of the $E_8$-term are similarly close to 0.

%

\subsection{Concentrating Volume via Automorphisms}  To determine the symplectic cone in the following sections, we will use the aforementioned automorphisms to map a given class of the form \ref{e:en} to one with certain properties.  Of particular interest will be the ability to control the volumes of certain components in $\alpha$.  In this section, we describe some of these methods.  

The following result is contained in \cite{H2}:

\begin{lemma}\label{l:ham}(Prop 2.10, \cite{H2}) Let $M$ be an elliptic surface and $\alpha\in mE_8\oplus kH$, $k\ge 2$, an integral class.  Then there exists a self-diffeomorphism of $M$ which maps $\alpha$ to
\[
\tilde \alpha=aA+bB\in H
\]
such that $a,b,\in\mathbb Z$, $\alpha^2=\tilde\alpha^2$ and both have the same divisibility.  This diffeomorphism is identity on the $(F,\Gamma)$-component.  If $M=E(2)$, then $\alpha$ maps to any other class of the same square and divisibility.

\end{lemma} 

\begin{example}  Consider the following, which illustrates how this is actually achieved:  Let $A=4\mathcal R_1+13T_1+7\mathcal R_2+9T_2=(4,13,7,9)$.  This can be transformed as follows, where $i=..$ denotes a map from Lemma \ref{l:rp} applied to the given rim pair:
\[
(4,13,7,9)\stackrel{i=2}{\rightarrow}(4,31,-1,9)\stackrel{-2-ref.}{\rightarrow}(4,31,9,-1)\stackrel{i=2}{\rightarrow}
\]
\[
(4,29,1,-1)\stackrel{i=28}{\rightarrow}(4,1,-111,-1)\stackrel{-2-ref.}{\rightarrow}(1,4,-1,-111)\stackrel{i=-1}{\rightarrow}
\]
\[
(1,115,0,-111)\stackrel{-2-ref.}{\rightarrow}(1,115,-111,0)\stackrel{i=111}{\rightarrow}(1,115,0,0)
\]
If an integral $A$ has any $E_8$-components, then transform this part to have only even entries and now apply Lemma \ref{l:e8h} to reduce them to 0 using the entry 1 in the above vector.

In this way it is possible to concentrate the volume for an integral class in one rim-pair.

The key to this procedure is the following observation:  For any automorphism which changes a term by $a\mapsto a+ib$, it is possible to choose $i$ such that the new entry satisfies $|a+ib|\le \frac{|b|}{2}$ if the sign of $a+ib$ is irrelevant or  $|a+ib|\le |b|$ if the sign of $a+ib$ is relevant.

Note that if the sign is relevant, then it is possible that the process terminates with all entries identical, up to a sign.  It is in this case possible to make one pair of entries identically 0, this in part motivates Def. \ref{d:bal}.
\end{example}

When $A\in H_2(M,\mathbb R)$, then it is unlikely that this procedure will allow the coefficients for an $E_8$ or $H$ term to be reduced to be identically 0.  The aim is to show that nonetheless the volume can be concentrated in a similar fashion as in Lemma \ref{l:ham}.

First, using the reduction described above, the volume of an $H$-term can be reduced below any bound.

\begin{lemma}\label{l:bal}Let $M$ be an elliptic surface which has a $2H$ component in the intersection form.  Let $(A_1,B_1,A_2,B_2)$ generate this $2H$ term in the intersection form.  Assume $(a_1,b_1,a_2,b_2)$ is not a multiple of an integer class.
Then there exists an automorphism of $H_2(M,\mathbb Z)$ acting only on this $2H$-component and by identity otherwise such that the following holds:  For every $k\in\mathbb N$ this class is equivalent to a class $(\tilde a_1,\tilde b_1,\tilde a_2,\tilde b_2)$ such that $|\tilde b_1|\le \frac{|b_1|}{2^{k-1}}$ and either  \begin{enumerate}
 \item $\tilde a_2=\tilde b_2=0$ or
 \item $0<|\tilde b_2|\le \frac{|b_1|}{2^k}$ and either
 \begin{enumerate}
 \item $0\le |\tilde a_2|\le  \frac{|b_1|}{2^k}$ if the sign of $\tilde a_2\cdot \tilde b_2$ is irrelevant or
 \item $|\tilde a_2|\le \frac{|b_1|}{2^{k-1}}$  and $\tilde a_2\cdot \tilde b_2>0$.
 \end{enumerate}
 \end{enumerate}
 
\end{lemma}

Note there is no control on the term $\tilde a_1$.  This is to be expected, as this term must account for the volume of the initial class, i.e. $2\tilde a_1\cdot\tilde b_1$ must carry an increasing amount of the initial volume, even though $\tilde b_1$ is decreasing.  On the other hand, the volume of the second rim-pair can be decreased below any given bound.

\begin{proof}
Using interchange maps and reflections on $-2$-classes, rewrite the initial class as 
\[
(b^0_1,a^0_1, a^0_2,b^0_2)
\]
with $0<|b^0_1|< |b^0_2|$.  If this is not possible, then the initial class is equivalent to a multiple of $(\pm 1,\pm 1,\pm 1,\pm 1)$, $(\pm 1,\pm 1,\pm 1,0)$ or $(\pm 1,0,\pm 1,0)$ or is of the form $(a,b,0,0)$. The first three cases have been excluded by assumption.

In the case $(a,b,0,0)$, assume that $|a|<|b|$.  Then this class can be mapped as follows:
\[
(a,b,0,0)\rightarrow(a,b,0,a)\rightarrow (a,b+ia,-ia,a)\rightarrow(b+ia,a,a,-ia) .
\]
It is then possible to choose $i\in\mathbb Z$ such that $0<|b+ia|\le \frac{|a|}{2}$.  Define this class to be $(b^0_1,a^0_1, a^0_2,b^0_2)$ with $0<|b^0_1|< |b^0_2|$.

Fix $d=|b_1^0|$.  Now apply Lemma \ref{l:rp} and $-2$-reflections, as in the example above, to obtain a class
\[
\left(\begin{matrix}b^0_1\\ a^0_1\\  a^0_2\\ b^0_2\end{matrix}\right)\rightarrow \left(\begin{matrix}b_1^0\\ a_1^0-ib_2^0-ja_2^0-ijb_1^0\\a_2^0+ib_1^0\\ b_2^0+jb_1^0\end{matrix}\right)=\left(\begin{matrix}b^1_1\\ a^1_1\\  a^1_2\\ b^1_2\end{matrix}\right).
\]
Note that $b_1^1=b_1^0$.  Choose $i,j\in\mathbb Z$ to reduce the second pair of coefficients.  The result is one of the following: \begin{enumerate}
\item $a^1_2=b^1_2=0$ or
\item $a^1_2=0$ and $0<|b_2^1|\le \frac{|b_1^0|}{2}=\frac{d}{2}$ or
\item $0<|a_2^1|,|b_2^1|\le \frac{d}{2}$ or
\item $a_2^1\cdot b_2^1>0$, $|b_2^1|\le \frac{d}{2}$ and $|a_2^1|\le d$.
\end{enumerate}

In each case this satisfies the claim of the Lemma for $k=1$.  If the result is in one of the last three cases, then the procedure can be continued by interchanging the two pairs and repeating this step on $(b^1_2,a^1_2, -a^1_1,-b^1_1)$.    Note that $0<|b^1_2|< |b^1_1|$.  In the first case, apply procedure for the $(a,b,0,0)$-case and then repeat the previous step using the class thus obtained.

%
%

\end{proof}

This result is central to achieving the goal of concentrating the volume in a similar fashion as Lemma \ref{l:ham}.  The following result will be at the core of the symplectic cone arguments.

\begin{theorem}\label{t:mine1}Let $M$ be an elliptic surface which has a $E_8\oplus 2H\oplus \langle F, \Gamma\rangle$ component in the intersection form. Let $\alpha_0$ with $\alpha_0\cdot F$ denote the $E_8\oplus 2H\oplus \langle F, \Gamma\rangle$-component of some class in $H_2(M,\mathbb R)$.  Then there exists an automorphism of $H_2(M,\mathbb Z)$ acting only on this $E_8\oplus 2H\oplus \langle F, \Gamma\rangle$-component and by identity otherwise such that one of the following two situations occurs:

\begin{enumerate}
\item If in $\alpha_0$ the $2H$ coefficients are not a multiple of an integral class, then for every $\epsilon>0$, $\alpha_0$ is equivalent to
\[
\alpha=\alpha_8+a_{1}A_{1}+b_{1}B_{1}+a_{2}A_{2}+b_{2}B_{2}+cF+g\Gamma
\]
with 
\begin{enumerate}
\item $0<2a_i\cdot b_i<\epsilon$ or $a_i=b_i=0$,
\item the coefficients $k_i$ of $\alpha_8$ satisfy $|k_i|< \epsilon$, $k_0<0$ and $k_i\ge 0$ for $i\ge 1$ and
\item $-\epsilon< \alpha_8^2\le 0$.

\end{enumerate}
In particular, the volume of $\alpha_0$ is concentrated in the term $cF+g\Gamma$.

\item If in $\alpha_0$ the $2H$ coefficients are a multiple of an integral class, then  either the volume is concentrated in $cF+g\Gamma$ as above or $\alpha_0$ is equivalent to
\[
\alpha=a_{1}A_{1}+b_{1}B_{1}+cF+g\Gamma.
\]

\end{enumerate}
This last case can only occur if the $E_8\oplus 2H$ terms are a multiple of an integral class.  Note further that in either case, the magnitude of the $E_8$-terms is reduced below $\epsilon$.
\end{theorem}

\begin{proof}  The aim of this proof is to use Lemma \ref{l:bal} and the automorphism of Lemma \ref{l:e8h} to decrease the magnitude of the corresponding coefficients as much as possible.  Note that Lemma \ref{l:e8h} implies that the class $\sum_{i=0}^7k_iD_i+aA+bB\in E_8\oplus H$ is equivalent to
\[
\tilde \alpha=\sum_{i=0}^7\tilde k_i D_i+\tilde aA+bB
\]
such for each $\tilde k_i$ one of the following holds:\begin{enumerate}
\item The sign of $\tilde k_i$ cannot be chosen freely and $|k_i|\le 
|b|$ or
\item the sign of $\tilde k_i$ can be pre-determined and $|k_i|\le 2|b|$.
\end{enumerate}
The key issue is the case excluded in Lemma \ref{l:bal} and this will be handled in cases.

{\bf Case 1:}  Assume that in $\alpha_0$ the $2H$ coefficients are not a multiple of an integral class.  Then Lemma \ref{l:bal} is applicable and use it to minimize one of the $H$-pair volumes.  This concentrates the volume of the $2H$-terms in a class $\beta=aA_1+bB_1$ with $0<|b|<\frac{\epsilon}{2g}$, achieved by choosing $k$ large enough.  Note that even if $\alpha_{2H}^2=0$, the fact that it is not a multiple of an integral class precludes it being equivalent to $(0,0,0,0)$, hence such a non-trivial $b$ must exist.  Use $\beta$ to minimize the coefficients of $\alpha_8$.  This leaves $b$ unchanged while changing $a$.  For large enough $k$ in Lemma \ref{l:bal}, it is thus possible to obtain a class
\[
\tilde\alpha_8+\tilde aA_{1}+bB_{1}+\tilde a_{2}A_{2}+\tilde b_{2}B_{2}+cF+g\Gamma
\]	
with \begin{enumerate}
\item $-\epsilon<\alpha^2_8\le 0$ and
\item $0\le 2\tilde a_2\cdot \tilde b_2<\epsilon$ and $|\tilde a_2|, |\tilde b_2|<\frac{\epsilon}{4g}\le \frac{\epsilon}{2}$ .
\end{enumerate}

If either $2\tilde a_2\cdot \tilde b_2>0$ or $\tilde a_2=\tilde b_2=0$ holds, then the class $(\tilde a_2, \tilde b_2)$ is of the required form.

This leaves the case that $2\tilde a_2\cdot \tilde b_2=0$ and $\tilde b_2\ne 0$. In this case, Lemma \ref{l:bal} also ensures that $0<|b|<\frac{\epsilon}{2g}$.   Apply the map from Lemma \ref{l:rp} one more time:
\[
(\tilde a, b,0,\tilde b_2)\rightarrow (b,\tilde a,0,\tilde b_2)\stackrel{Lemma \ref{l:rp}}{\rightarrow} (b, \tilde a+i\tilde b_2,-ib,\tilde b_2).
\] 
Choose $i\in\{\pm 1\}$ such that $-ib$ and $\tilde b_2$ have the same sign.  

This shows that either $0<2\tilde a_2\cdot \tilde b_2<\epsilon$ or $\tilde a_2=\tilde b_2=0$.

In this way obtain a class that has the volume concentrated in
\[
cF+g\Gamma+\tilde aA_{1}+bB_{1}
\]
with $0<|b|<\frac{\epsilon}{2g}$.  (Observe the similarity to this setup and the result in Lemma \ref{l:ham}.) 

Now apply the map from Lemma \ref{l:fi} to $(\tilde a,b)$ to obtain
\[
(c,g,\tilde a, b)\rightarrow (c,g, b, \tilde a)\rightarrow \left(c-\frac{i}{\tau}b,g,b,\tilde a+ig\right).
\]
Choose $i$ such that $b$ and $\tilde a+ig$ have the same sign and $0<|\tilde a+ig|\le g$.  Then 
\[
0<2b\cdot |\tilde a+ig|<\epsilon.
\]

{\bf Case 2:} Assume that in $\alpha_0$ the $2H$ coefficients are a multiple of an integral class.  The goal is to show that either the previous case can be applied or that the $E_8\oplus 2H$ coefficients are a multiple of an integral class, hence Lemma \ref{l:ham} can be applied. Write 
\[
\alpha_0=(k_0,...,k_7, p,q,r,s,c,g)
\]
and let $(p,q,r,s)=\omega (p^\mathbb Z,q^\mathbb Z,r^\mathbb Z,s^\mathbb Z)$ for some non-zero $\omega\in\mathbb R$.  Here $*^\mathbb Z\in\mathbb Z$  corresponds to $*$, i.e. $p^{\mathbb Z}\in\mathbb Z$ and $p=\omega p^\mathbb Z$.

{\bf Case 2.1:}  Assume that $\frac{1}{\omega} g\not\in \mathbb Q$ and that $(p,q,r,s)$ is not identical to the zero vector.  Let $p\ne 0$.  Then one application of an $H$-fiber map changes the $2H$-terms to $(p,q,r,s+g)$.  Assume this is still a multiple of an integer class.  Then there exists some non-zero $\tilde \omega\in\mathbb R$ such that 
\[
(p,q,r,s+g)=\tilde\omega\left(\tilde p^\mathbb Z,\tilde q^\mathbb Z,\tilde r^\mathbb Z,\tilde s^\mathbb Z+\frac{1}{\tilde\omega} g\right).
\]
 Then $p=\omega p^\mathbb Z=\tilde\omega\tilde p^\mathbb Z$ and thus
 \[
 \frac{\tilde\omega}{\omega}\in\mathbb Q.
 \]
This implies that 
\[
\frac{1}{\omega}g=\frac{\tilde\omega}{\omega}\cdot\frac{1}{\tilde\omega}g\in\mathbb Q.
\] 
Hence $\alpha_0$ is equivalent to a class for which the $2H$ coefficients are not a multiple of an integral class, now Case 1 applies and the result follows.
 
{\bf Case 2.2:}  Assume now that $\frac{1}{\omega} g\in \mathbb Q$ or that $(p,q,r,s)$ is identical to the zero vector.  Now apply Lemma \ref{l:e8h} to the first ten entries of either the class
\[
(k_0,...,k_7,p,q,r,s)=\omega\left(\frac{k_0}{\omega},...,\frac{k_7}{\omega},p^\mathbb Z,q^\mathbb Z,r^\mathbb Z,s^\mathbb Z\right)
\]
or, after one application of the $H$-fiber map to obtain the class $(0,0,0,0)\rightarrow (0,0,0,g)$, the class
\[
(k_0,...,k_7,0,0,0,g)=g\left(\frac{k_0}{g},...,\frac{k_7}{g},0,0,0,1\right).
\]
In either case, Lemma \ref{l:e8h} produces a class with
\[
p^\mathbb Z\mapsto p^\mathbb Z+4q^\mathbb Z\sum r_i +2\left(r^T\cdot E_8\cdot \frac{1}{\omega}k\right)
\]
where the first two terms in the sum are integers ($p^\mathbb Z,q^\mathbb Z=0$ is allowed).  Assume that for some choice of $r$ the term  $r^T\cdot E_8\cdot \frac{1}{\omega}k\not\in\mathbb Z$.  Then the class obtained for this choice of $r$ will have a $2H$ term which is not a multiple of an integral class and thus again Case 1 applies.

Assume that for all choices of $r\in\mathbb Z^8$, the term 
\[
r^T\cdot E_8\cdot \frac{1}{\omega}k\in\mathbb Z.
\]
This in particular holds for $r=e_i$ one of the eight basis vectors of $\mathbb Z^8$.  Using these eight, the condition can be rewritten as
\[
\left(\begin{matrix}e_0\\ e_1\\\vdots\\e_7\end{matrix}\right)\cdot E_8\cdot\frac{1}{\omega}k=Id_{8}\cdot E_8\cdot\frac{1}{\omega}k=E_8\cdot \frac{1}{\omega}k\in\mathbb Z^8.
\]
As $E_8$ is invertible over the integers, this implies that $\frac{1}{\omega}k\in\mathbb Z^8$.  Hence 
\[
\alpha_0=\omega\left(k_0^\mathbb Z\,...,k_7^\mathbb Z\,p^\mathbb Z,q^\mathbb Z,r^\mathbb Z,s^\mathbb Z,\frac{1}{\omega}c,\frac{1}{\omega}g\right).
\]
Then apply Lemma \ref{l:ham} to obtain a class $\alpha=(0,...,0,a_1, b_1,0,0, c, g)$ equivalent to $\alpha_0$.

\end{proof}

In particular, the volume has been concentrated in the $(A_1, B_1,F,\Gamma)$-term.  

\subsection{Balancing Classes in Elliptic Surfaces}

The methods of concentrating volume from the previous section will now be applied to produce balanced classes in elliptic surfaces of positive Euler number.  The following result covers most, but not all elliptic surfaces.  For $E(1,g)$, balanced will not be relevant (see  Theorem \ref{t:e1}) and $E(2,g)$ will be addressed in Theorem \ref{t:e2}.  The remaining surfaces have multiple fibers and will be addressed in a subsequent paper.

\begin{lemma}\label{l:ba}
Let $M$ be a relatively minimal elliptic surface with positive Euler characteristic $\chi(M)=12n$ and not diffeomorphic to $E(n,0,p_1,...,p_k)$, $n\in\{1,2\}$, and $\alpha_0\in \mathcal P_M^{F}$.  Suppose $M=X_1\#_{\tilde F_g}X_2$ is obtained as the fiber sum of elliptic surfaces $X_i$.  Then $\alpha_0$ is equivalent to a class $\alpha\in\mathcal P^F_M$ and $\alpha$ is balanced with respect to $(X_1,X_2)$.

In particular, $\alpha$ can be chosen so that for $\epsilon>0$, each term $k_i$ in each of the $E_8$-components satisfies $0\le k_i<\epsilon$. 
\end{lemma}

\begin{proof}  If the given fiber sum decomposition admits no rim pairs, then the class $\alpha_0$ is already balanced.  In particular, this is the case if $(n,g)=(1,\ge 1)$.  Furthermore, fix any two $H$-terms in the class $\alpha_0$ and apply Theorem \ref{t:mine1} to minimize each $E_8$-term.

Assume now that the sum has rim components $(\mathcal R_1,T_1,\mathcal R_2,T_2)$.  If $\chi(M)=12n$, then $M$ has intersection form given by 
\[
nE_8\oplus  \left[2(n-1)+2g\right]H\oplus \left(\begin{matrix} 0&1\\1 &-n\end{matrix}\right).
\]
Using \eqref{e:en}, write
\[
\alpha_0=\sum_{i=1}^n\alpha_{8,i}+\sum_{j=1}^{2n+2g-4}\alpha_{H,j}+\alpha_{RT}+\alpha_F.
\]
Note that in the remaining cases, $(n,g)=(\ge 2, \ge 1)$ and $(n,g)=(> 2, 0)$ we have $2n+2g-4\ge 2$.

In the class $\alpha_0$, apply Theorem \ref{t:mine1} to the component  $\alpha_{8,1}+\alpha_{RT}+\alpha_F$.  This produces an equivalent class $\alpha_1$ which satisfies one of the following:\begin{enumerate}
\item The rim components and the entries of $\tilde\alpha_{8,1}$ are minimized for the given $\epsilon$, i.e.
\[
\alpha_1=\underbrace{\tilde\alpha_{8,1}}_{0\le |k_i|<\epsilon}+\sum_{i=2}^n\alpha_{8,i}+\sum_{j=1}^{2n+2g-4}\alpha_{H,j}+\underbrace{\alpha_{RT}}_{<\epsilon}+\tilde c F+g\Gamma.
\] 
In particular, the class $\alpha_1$ is balanced with respect to $(X_1,X_2)$. 

To minimize the remaining $\alpha_{8,i}$ terms, choose a pair of $H$-terms $\alpha_{H,1}+\alpha_{H,2}$, which exist as $2n+2g-4\ge 2$, and  apply Theorem \ref{t:mine1} to the classes $\alpha_{8,i}+\alpha_{H,1}+\alpha_{H,2}+\tilde cF+g\Gamma$ for $i\ge 2$ ($\alpha_{8,i}+\alpha_{H,1}+\alpha_{H,2}$ are unchanged in the map from $\alpha_0$ to $\alpha_1$).  This produces an equivalent class, with $\alpha_{RT}$ unchanged to $\alpha_1$,
\[
\alpha=\underbrace{\sum_{i=1}^n\alpha_{8,i}}_{\mbox{each }0\le |k_i|<\epsilon}+\sum_{j=1}^{2n+2g-6}\alpha_{H,j}+\tilde\alpha_{H,1}+\tilde\alpha_{H,2}+\underbrace{\alpha_{RT}}_{<\epsilon}+\tilde{\tilde c} F+g\Gamma
\] 
which has the $\alpha_{8,i}$ terms each minimized and is balanced with respect to $(X_1,X_2)$.

\item Case 2 of Theorem \ref{t:mine1} may produce an equivalent class of the form
\[
\alpha_1=\sum_{i=2}^n\alpha_{8,i}+\sum_{j=1}^{2n+2g-4}\alpha_{H,j}+a\mathcal R_1+bT_1+0\mathcal R_2+0T_2+\alpha_F.
\]
This class may not be balanced with respect to $(X_1,X_2)$ as the Theorem allows no control over the coefficients $(a,b)$.

As before, choose a pair of $H$-terms $\alpha_{H,1}+\alpha_{H,2}$ and  apply Theorem \ref{t:mine1} to the classes $\alpha_{8,i}+\alpha_{H,1}+\alpha_{H,2}+\tilde cF+g\Gamma$ for $i\ge 2$.  This again produces an equivalent class
\[
\alpha_2=\underbrace{\sum_{i=1}^n\alpha_{8,i}}_{\mbox{each }0\le |k_i|<\epsilon}+\sum_{j=1}^{2n+2g-6}\alpha_{H,j}+\tilde\alpha_{H,1}+\tilde\alpha_{H,2}+a\mathcal R_1+bT_1+0\mathcal R_2+0T_2+\tilde{\tilde c} F+g\Gamma
\] 
where either each term of $\tilde\alpha_{H,i}=a_iA_i+b_iB_i$ has magnitude $|a_i|,|b_i|<\epsilon$ or
\[
\tilde\alpha_{H,1}+\tilde\alpha_{H,2}=0A_1+0B_1+a_2A_2+b_2B_2.
\]
Now apply an interchange map to the classes $\tilde\alpha_{H,1}$ and $a\mathcal R_1+bT_1$ to obtain an equivalent class
\[
\alpha=\underbrace{\sum_{i=1}^n\alpha_{8,i}}_{\mbox{each }0\le |k_i|<\epsilon}+\sum_{j=1}^{2n+2g-6}\alpha_{H,j}+aA_1+bB_1+\tilde\alpha_{H,2}+\tilde\alpha_{RT}+\tilde{\tilde c} F+g\Gamma
\]
with $\tilde\alpha_{RT}=(-a_1)\mathcal R_1+(-b_1)T_1+0\mathcal R_2+0T_2$ where $|a_1|,|b_1|<\epsilon$.  Hence this class is balanced with respect to $(X_1,X_2)$ and each of the $\alpha_{8,i}$ terms has been minimized.
\end{enumerate}

\end{proof}

Observe that this result shows that given any $\alpha_0\in \mathcal P_M^{F}$ and any decomposition $M=X_1\#_{F_g}X_2$, there is a class $\alpha$ balanced with respect to $(X_1,X_2)$.  In particular, there is no need to further specify $(X_1,X_2)$.  {\it Hence, in the following, a class will simply be referred to as balanced.}

\begin{cor}\label{c:bl}  Let $N$ be the blow up of a relatively minimal elliptic surface $M$ with positive Euler characteristic $\chi(M)=12n$ and not diffeomorphic to $E(n,0,p_1,...,p_k)$, $n\in\{1,2\}$, ${F_g}$ a generic fiber and $\alpha_0-\sum_{i=1}^l e_iE_i\in \mathcal C_{N,K}^{F}$.  Suppose $N=X_1\#_{F_g}X_2$ is obtained as the fiber sum of elliptic surfaces $X_i$.  Then $\alpha_0-\sum_{i=1}^l e_iE_i$ is equivalent to a class $\alpha-\sum_{i=1}^l e_iE_i$ with $\alpha \in\mathcal P^F_M$ and $\alpha$ balanced with respect to $(X_1,X_2)$.

In particular, $\alpha$ can be chosen so that for $\epsilon>0$, each term $k_i$ in each of the $E_8$-components satisfies $0\le k_i<\epsilon$.  

\end{cor}

\begin{proof}
Theorem \ref{t:2} states that every diffeomorphism used in Lemma \ref{l:ba} is identity on a neighborhood of $F_g$.  The procedure in Lemma \ref{l:ba} uses finitely many such diffeomorphisms, hence their composition defines a diffeomorphism $\psi:M\rightarrow M$ that also fixes a neighborhood of $F_g$.  Choosing the blow-up locus to lie in this neighborhood ensures that the diffeomorphism $\psi$ can be applied to $N$ without changing the exceptional classes.

\end{proof}

\section{Symplectic Cones}\label{s:4}

Let $M$ be a smooth oriented 4-manifold.  A symplectic form $\omega$ on $M$ is a closed non-degenerate 2-form.  As $M$ is oriented, it is natural to restrict to forms $\omega$ compatible with the fixed orientation.  This means $\omega\wedge\omega>0$, or $[\omega]^2>0$.  Hence automatically $[\omega]\in\mathcal P_M$.  

\begin{definition}Let $M$ be a smooth oriented $4$-manifold and $V\subset M$ a smooth oriented submanifold.  \begin{enumerate}
\item Define the {\it symplectic cone of $M$} to be
\[
\mathcal{C}_M=\{\alpha\in H^2(M,\mathbb R)\;|\;[\omega]=\alpha, \; \omega\mbox{ is a symplectic form on }M\}. 
\]
\item A relative symplectic form on the pair $(M, V)$ is
an orientation compatible symplectic form on $M$ such that
$\omega\vert_V$ is an orientation compatible symplectic
form on $V$. 

\item The relative symplectic cone of $(M, V)$ is\vspace*{3pt}
\[
{\mathcal C}_M^V=\{\alpha\in H^2(M)|\;
[\omega]=\alpha,\;\omega \mbox{ is a relative symplectic form on $(M, V)$}\}.
\]

\item The cone of symplectic classes evaluating positively on $[V]$ is
\[
{\mathcal C}_M^{[V]}=\{\alpha\in\mathcal{C}_M\;|\;\alpha\cdot [V]>0\}.
\]
\end{enumerate} 
\end{definition}

The comments preceding the definition imply $\mathcal C_M\subset \mathcal P_M$.  Moreover, for $V$ to be $\omega$-symplectic, we must have $\omega|_V$ is a volume form or $[\omega]\cdot [V]>0$.  Hence
\begin{equation}\label{e:incl}
\mathcal C_M^V\subset{\mathcal C}_M^{[V]}\subset \mathcal P_M^{[V]}.
\end{equation}

If $M$ is non-minimal, then the exceptional curves provide further constraints on the symplectic classes.  Denote the following:\begin{enumerate}
\item $\mathcal E_M$ the set of cohomology classes whose Poincar\'e dual are represented by smoothly embedded spheres of self-intersection -1,
\item  $\mathcal K$ the set of symplectic canonical classes of $M$ and 
\item for $K\in\mathcal K$,
\[
\mathcal E_K=\{E\in \mathcal E_M|\;K\cdot E=-1\}.
\]
\end{enumerate}

\subsection{Relative Symplectic Cones for $ b^+(M)=1$}  To motivate this discussion, a first result for elliptic surfaces with $b^+=1$, irrespective of Euler number, is stated.  This is a consequence of Theorem 2.13, \cite{DL}.

\begin{theorem}\label{t:b10} Assume that $M$ is an elliptic surface with $b^+=1$ and $F_g$ is an oriented
 generic fiber such that $\mathcal C^{F_g}_M\ne \emptyset$.  Denote $F=[F_g]\in H_2(M,\mathbb Z)$.  Let 
\[
\mathcal K(F_g)=\{K\in\mathcal K\;|\;K\cdot F=0\}
\]
be the set of symplectic canonical classes of $M$ which evaluate to 0 on $F$.  For each $K\in\mathcal K(F_g)$, let $w$ be a symplectic form with $K_\omega=K$ and define $ \mathcal P^F_{M,+}$ to be the component of $\mathcal P^F_M$ containing $[\omega]$.  Then
\[
\bigsqcup_{K\in\mathcal K(F_g)}\mathcal C_{M,K}^F=\mathcal C_M^{F_g}
\]
where 
\[
\mathcal C_{M,K}^F=\{\alpha\in\mathcal P^F_{M,+}\;|\; \alpha\cdot E>0\;\forall E\in\mathcal E_K\}.
\]

\end{theorem}

If $b^+(M)=1$, the cone $\mathcal P^F_M$ is connected (Lemma 2.2, \cite{DL}) and thus must lie in one of the two connected components of $\mathcal P_M$.

%
%
%

This implies, that if $K\in \mathcal K(F_g)$, then only one of $\mathcal C_{M,K}^F$ or $\mathcal C_{M,-K}^F$ is non-empty.  
This motivates Def. \ref{d:cmk}.

\begin{cor}\label{c:min}
Assume $M$ is as in Theorem \ref{t:b10}.  Assume further that $M$ is minimal. Then 
\[
\mathcal C_M^{F_g}=\mathcal C_{M,K_{min}}^{F}=\mathcal P^F_M
\]
where $K_{min}$ is given by \eqref{e:k}.
\end{cor}    

\subsubsection{$T^2\times S^2$}\label{s:s2t2}
Viewing this as an elliptic surface, let $\Gamma_{S^2}$ be the section, represented by a sphere.  This manifold has intersection form $H$ and it follows from Cor. \ref{c:min} that
\[
\mathcal C_{T^2\times S^2}^{F_g}=\mathcal P^F_{T^2\times S^2}.
\]
In fact
\[
\mathcal C_{T^2\times S^2, -2F}^{F_g}=\{aF+b\Gamma_{S^2}\;|\;a,b>0\}\mbox{ and  } \mathcal C_{T^2\times S^2, 2F}^{F_g}=\emptyset.
\]  
Consider now the non-minimal case.  Let $M=(T^2\times S^2)\#l\overline{\mathbb CP^2}$.  For each blow up, two exceptional spheres are generated with classes $E_i$ and $\Gamma-
E_i$.  The set $\mathcal K(F_g)$ is given by 
\[
\mathcal K(F_g)=\{\pm 2F\pm E_1\pm...\pm E_l\}.
\]
For each $K\in\mathcal K(F_g)$, let $\delta_i=K\cdot E_i$.  Then 
\[
\mathcal E_K=\{ -\delta_1E_1, \Gamma_{S^2}+\delta_1E_1,...,-\delta_lE_l,\Gamma_{S^2}+\delta_lE_l\}.
\]
The light cone lemma implies that
\[
\mathcal C_{T^2\times S^2\#l\overline{\mathbb CP^2}, 2F\pm E_1\pm...\pm E_l}^{F}=\emptyset.
\]

Hence, as a consequence of Theorem \ref{t:b10} and \cite{Bi2}, for the symplectic canonical class $-2F+\sum  E_i$, a class $\beta=aF+b\Gamma_{S^2}-\sum e_iE_i$ satisfying
\[
\beta^2>0,\;\;\beta \cdot F>0,\;\;\beta\cdot  E_i>0 \mbox{  and  }\beta\cdot (\Gamma_{S^2}- E_i)>0.
\]
lies in $\mathcal C_{T^2\times S^2\#l\overline{\mathbb CP^2}}^{F_g}$ and hence can be represented by a symplectic form making $ F_g$ symplectic.  

Similar results hold for $E(1)$, see Theorem \ref{t:be1}, but the set $\mathcal E_K$ is much larger and thus ensuring that $\alpha\cdot E>0$ is much more involved, see Lemma \ref{l:e1setup} and the proof of Theorem \ref{t:e1}.

\subsubsection{$\kappa(M)\ge 0$}  Assume first that $M$ is a relatively minimal elliptic surface.    Then by Cor \ref{c:min}, 
\[
\mathcal C^{F_g}_{M}=\mathcal C_{M,K_{min}}^F=\left\{\begin{array}{cc}\mathcal P_M&\kappa(M)=0,\\\mathcal  P^F_{M}&\kappa(M)=1.\end{array}\right.
\]

If $M=M_{min}\#l\overline{\mathbb CP^2}$, then it was shown in \cite{TJLL} that the set of symplectic canonical classes is given by 
\[
\mathcal K=\{\pm K_{min}\pm E_1\pm...\pm E_l\}.
\]
For each $K\in\mathcal K$, let $\delta_i=K\cdot E_i$.  Then, as $\kappa(M)\ge 0$, 
\[
\mathcal E_K=\{ -\delta_1E_1,...,-\delta_lE_l\}.
\]
If $F_g$ is a generic fiber of the elliptic fibration, then
\[
\mathcal C_{M, -K_{min}\pm E_1\pm...\pm E_l}^{F}=\emptyset.
\]

\begin{definition}\label{d:cmk}
Let $M$ be an elliptic surface with $\kappa(M)\ge 0$ and $F_g$ an oriented generic fiber.  \begin{enumerate}
\item Denote by
\[
\mathcal K_F=\{K_{min}\pm E_1\pm...\pm E_l\}\subset \mathcal K(F_g)
\]
the set of admissible symplectic canonical classes for $F_g$.
\item   Let $K\in\mathcal K_F$.  Then define
\[
\mathcal C_{M,K}^F=\{\alpha\in\mathcal P^F_{M}\;|\; \alpha\cdot E>0\;\forall E\in\mathcal E_K\}.
\]
\end{enumerate}
\end{definition}

It follows from Lemma 3.5, \cite{TJLL} that
\[
\mathcal C_M^{F_g}\subset \bigsqcup_{K\in\mathcal K(F_g)}\mathcal C_{M,K}^F.
\]

Guided by the $b^+=1$ result above and in light of \ref{e:incl}, it needs to be shown that for each $K\in\mathcal K_F$, 
\begin{equation}\label{e:K}
\mathcal C_{M,K}^F\subset \mathcal C_M^{F_g},
\end{equation}
while noting that in the relatively minimal case this is just the inclusion $\mathcal P_M^F\subset \mathcal C_M^{F_g}$. 

\subsection{The Relative Symplectic Cone and Fiber Sums}
Let $M=X\#_VY$, then if $X$ and $Y$ are symplectic manifolds, it was shown by Gompf \cite{Go2} (see also McCarthy-Wolfson \cite{MW}) that $M$ admits a symplectic structure.  Thus, ideally, to determine $\mathcal C^V_M$ one would "add" the relative symplectic cones $\mathcal C_X^V$ and $\mathcal C_Y^V$.  In the absence of rim tori, this was done in \cite{DL}.

\begin{theorem}\cite{DL}\label{t:38}  Suppose $M=X\#_VY$, the sum produces no rim components and $V$ has trivial normal bundle.  If $\mathcal C_*^V=\mathcal P_*^{[V]}$ holds on $X$ and $Y$, then $\mathcal C_M^V=\mathcal P_M^{[V]}$. 
\end{theorem}

The proof of this claim decomposes a class $\alpha\in\mathcal P_M^{[V]}$ into two classes $\alpha_X$ and $\alpha_Y$.  These both evaluate positively on $[V]$, but care must be taken to ensure that they square positively.  This can always be achieved by choosing the coefficient of $[V]$ in each term appropriately.  Then the claim follows by \cite{Go2}.   

The aim of the remainder of this section is to extend this result to include exceptional curves and rim components.  More specifically,
\begin{enumerate}
\item Theorem \ref{t:exc} allows for exceptional curves, but only finitely many and in $X$ and $Y$ these are disjoint from a neighborhood of the gluing fiber.  It does not allow for rim components.
\item Theorem \ref{t:s2t2} still does not allow rim components, but does allow one summand to admit finitely many exceptional curves, some of which may intersect the gluing fiber.
\item Theorem \ref{t:sym} finally deals with the rim components.  It does not address the presence of exceptional curves, these are hidden in the assumptions, should they be present.  This Theorem also requires the full use of the concept of a balanced class.
\end{enumerate}

\begin{theorem}\label{t:exc}
Suppose $M=X_1\#_{\tilde F_g}X_2$ is an elliptic surface and $F_g$ a generic smooth fiber.  Assume $M$ admits finitely many $K_M$-exceptional curves $E_1,..,E_l$ such that $E_i\cdot F=0$. Assume further the following:
\begin{enumerate}
\item The sum produces no rim components,
\item $X_i$ is non-minimal with finitely many exceptional $K_{X_i}$-curves $E_1,...,E_{l_i}$ with $E_i\cdot F=0$ (and disjoint from $\tilde F_g$) and
\item   $\mathcal C_{X_i,K_{X_i}}^F\subset \mathcal C_{X_i}^{F_g}$.

\end{enumerate}

Then $\mathcal C_{M,K}^F\subset \mathcal C_M^{F_g}$. 
\end{theorem}

\begin{proof}
The proof is identical to the proof of Theorem \ref{t:38} in concept.  Consider a class $\alpha\in \mathcal C_{M,K}^F$.  This can be written as
\[
\alpha=(\alpha_1-\sum_{i=1}^{l_1}e_iE_i)+cF+g\Gamma + (\alpha_2-\sum_{i=1}^{l_2}e_iE_i).
\]
This is decomposed into two classes as in \cite{DL} as
\[
\alpha_{X_i}=\alpha_i+c_iF+g\Gamma_i-\sum_{i=1}^{l_1}e_iE_i
\]
which satisfy $\alpha_{X_i}\cdot E_i=e_i>0$.  To ensure that $\alpha_{X_i}^2>0$ holds, choose $c_1$ so that $\alpha^2>\alpha_{X_1}^2>0$.  Then $c_2=c-c_1$ will ensure that $\alpha_{X_2}^2>0$.  The claim now follows as in \cite{DL}.
\end{proof}

\begin{theorem}\label{t:s2t2}
Suppose $M=M_m\#l\overline{\mathbb CP^2}$ is an elliptic surface and $F_g\subset M$ a generic smooth fiber.  Fix a symplectic canonical class $K\in\mathcal K_F$ on $M$ and  let $\mathcal E_K=\{E_1,..,E_l\}$.  Assume that $M_m$ is minimal with respect to these exceptional classes and that $\mathcal C_{M_m}^{F_g}=\mathcal P_{M_m}^F$.

Let $\alpha=\alpha_m+cF+g\Gamma-\sum_{i=1}^le_iE_i\in \mathcal C_{M,K}^F$.  Assume for all \[
0<\epsilon<<\min\{1,\alpha^2-\sum e_i^2\}
\]
 there exists a diffeomorphism $\psi_{\epsilon}:M_m\rightarrow M_m$ such that $\alpha_m+cF+g\Gamma$ is mapped to a class $\alpha_m^++c_1\tilde F+ g_1\tilde\Gamma$ with $0<g_1<\epsilon$.  Then $\alpha\in \mathcal C_{M}^{F_g}$.

In particular, if for every $\alpha$ and small enough $\epsilon>0$  such a diffeomorphism exists, then $\mathcal C_{M,K}^F\subset \mathcal C_{M}^{F_g}$.

\end{theorem}

\begin{proof}  Let $K_m$ be the symplectic canonical class induced from $K$ on $M_m$ and $K_m^\epsilon$ the pull-back under $\psi_{\epsilon}^{-1}$ on $M_m$.

Consider the class $\alpha^+=\alpha_m^++c_1\tilde F+g_1\tilde\Gamma$ with $0<g_1<\epsilon$ and let $F_g^+=\psi_{\epsilon}(F_g)\subset M_m$.   Decompose $M_m$ as the trivial fiber sum 
\[
M_m=M_m\#_{\tilde F_g}(T^2\times S^2)
\]
such that $F_g^+$ lies in the $M_m$-summand on the left.  Decompose $\alpha^+$ into two classes
\begin{itemize}
\item $\alpha_{M_m}=\alpha_m^++(c-\tilde c)\tilde F+ g_1 \Gamma_m$ and
\item $\alpha_{T^2\times S^2}=\tilde c \tilde F+ g_1 \Gamma_{S^2}.$
\end{itemize}
This fiber sum also splits the symplectic canonical class $K^\epsilon_m=(K_m^+,K_{T^2\times S^2})$.  Choose $\tilde c>0$ such that 
\[
0<\alpha_{M_m}^2=(\alpha^+)^2-2\tilde cg_1 <\epsilon
\]
 and thus $\alpha_{M_m}\in \mathcal C_{M_m}^{F^+_g}$ (represented by a symplectic form $\omega_{M_m}$).  This implies that $\alpha_{T^2\times S^2}^2\simeq (\alpha^+)^2$ and  hence $\alpha_{T^2\times S^2}\in\mathcal C^{\tilde F_g}_{T^2\times S^2}$ by Theorem \ref{t:b10}.  

Consider the class $\alpha_{bu}=\alpha_{T^2\times S^2}-\sum_{i=1}^l e_i\tilde E_i$.  Then $\alpha_{bu}\cdot \tilde F=g_1>0$.  The choice of $\tilde c$ described above ensures that  $\alpha_{bu}^2>0$ and the initial condition on the $e_i$ ensures that $\alpha_{bu}\cdot\tilde E_i>0$.
Finally, consider 
 \begin{equation}\label{e:bu}
 \alpha_{bu}\cdot (  \Gamma_{S^2}-\tilde E_j)=\tilde c-e_j.
 \end{equation}
 Note that 
 \[
 \alpha_{bu}^2=2\tilde c \tilde g-\sum e_i^2>0\;\;\Rightarrow \;\;
\tilde c - e_j>\frac{1}{2\tilde g}\sum e_i^2-e_j.
\]
Hence choosing 
\[
\epsilon<\frac{1}{2\max\{e_j\}}\sum e_i^2
\]
ensures that $\alpha_{bu}\cdot (\Gamma_{S^2}-\tilde E_j)>0$ holds.  By Section \ref{s:s2t2}, this would imply that $\alpha_{bu}$ is represented by a symplectic form that makes $\tilde F_g$ symplectic.  

This means that there exists a symplectic form $\omega$ representing $\alpha_{T^2\times S^2}$ which makes $\tilde F_g$ symplectic and which can be blown up $l$-times of weights $e_i$ to obtain a symplectic form in  $T^2\times S^2\#l\overline{\mathbb CP^2}$ which still makes $\tilde F_g$ symplectic.

Blowing up $l$-points in $S^2\times T^2$ symplectically involves removing symplectically embedded balls $\psi_i:(B^2(e_i),\omega_{st})\rightarrow (S^2\times T^2,\omega)$ corresponding to the weight $e_i$ and gluing back in a standard neighborhood for each. Use the diffeomorphism $\psi_\epsilon$ to define symplectic embeddings of these balls in $M_m$ with respect to the form $(\psi_{\epsilon})_*(\omega_{M_m}, \omega)$ in the class $\alpha_m+cF+g\Gamma$.  Note that the fiber $F_g$ is disjoint from these embeddings due to the choice of splitting with respect to $F_g^+$.

This ensures it is possible to blow up the original $(M_m,\alpha_m+cF+g\Gamma)$ to obtain a symplectic form $\omega$ representing the class $\alpha$ and which makes $F_g$ $\omega$-symplectic.  Note that after blowing up, this construction ensures that a fibration part exists in $M$, albeit of very small volume.  Hence $\alpha\in\mathcal C_{M}^{F_g}$.

%
\end{proof}

{\bf Remark:}  As $T^2\times S^2$ admits a full packing with respect to the symplectic class $\alpha_{T^2\times S^2}$ constructed above (see \cite{S}), it is interesting to consider if the manifolds studied in this Theorem all admit full packings.

In the presence of rim components arising in the sum, the methods of proof of the previous Theorems no longer apply in general.  However, if the rim components are removed, then a sum class can be shown to be symplectically represented and from this the original class can be obtained.  This motivates the following definition.

\begin{definition} \label{d:sbal}Let $M=X\#_{F_g}Y$ be an elliptic surface and $\alpha\in\mathcal P^F_M$ a balanced class.  Then $\alpha$ is sum balanced with respect to $(X,Y)$ if  the class $\alpha-(e_1\mathcal R_1+d_1T_1+e_2\mathcal R_2+d_2T_2)$ can be written as $\alpha_X+(c_X+c_Y)F+g(\Gamma_X+\Gamma_Y)+\alpha_Y$ such that $\alpha_*+c_*F+g\Gamma_*\in\mathcal C^{F_g}_*$.

\end{definition}

Lemma \ref{l:ba} shows that for most elliptic surfaces any class in $\mathcal P_M^F$ is equivalent to a balanced class.  If  this class is actually sum balanced and has no rim components, then \cite{Go2} places the class in the relative symplectic cone.  In the presence of rim components, the idea is to start with symplectic sum form representing the class $\alpha-(e_1\mathcal R_1+d_1T_1+e_2\mathcal R_2+d_2T_2)$.  This class needs to be modified to account for the missing rim components.

The symplectic form can be modified using submanifolds of $M$.  In the presence of Lagrangian submanifolds, the symplectic form can be modified to obtain symplectic submanifolds, this is a modification of a result in \cite{H} and \cite{Go2}.

\begin{theorem}\label{t:lag}Let $(M,\omega)$ be a symplectic 4-manifold and $L_1, L_2$ closed connected embedded oriented Lagrangian surface in $M$ which intersect each other transversely and which generate a summand in the intersection form.  Suppose that the classes are linearly independent in $H_2(M,\mathbb R)$.  Then there exists a symplectic structure $\tilde \omega$ on $M$ with the following properties:\begin{enumerate}
\item $\tilde \omega$ is deformation equivalent to $\omega$, 
\item both $L_i$ are $\tilde\omega$-symplectic,
\item  $\tilde\omega$ can be chosen such that $[\tilde\omega]\cdot [L_i]=l_i$ have any given sign,
\item $[\tilde\omega]-l_1[L_1]-l_2[L_2]=[\omega]$ and 
\item any $\omega$-symplectic surface disjoint from the Lagrangians $L_i$ is $\tilde\omega$-symplectic.  
\end{enumerate}
Moreover, $\omega$ and $\tilde \omega$ differ only on a neighborhood of $L_1\cup L_2$.

\end{theorem}

\begin{proof}The proof of this theorem is essentially the same as the proof of Theorem 10, \cite{H}, however more care must be taken in choosing the class $\eta$ to ensure that the fourth claim holds.  To achieve this, let $\eta_i$ denote the Thom form of the submanifold $L_i$, this can be chosen to have support on any given tubular neighborhood of $L_i$ (Prop 6.25, \cite{BT}).  Then consider the closed 2-form $\eta(s,t)=u\eta_1+v\eta_2$ for $u,v\in\mathbb R$.  Given $a_1,a_2$ as in the \cite{H}, the system
\[
\left(\begin{matrix}\int_{L_1}\eta(u,v)\\\int_{L_2}\eta(u,v)\end{matrix}\right)=\left(\begin{matrix}\int_{L_1}\eta_1&\int_{L_1}\eta_2\\\int_{L_2}\eta_1 &\int_{L_2}\eta_2\end{matrix}\right)\left(\begin{matrix}u\\v \end{matrix}\right)=\left(\begin{matrix} a_1\\a_2\end{matrix}\right)
\]
has a solution whenever the matrix in the middle term has non-vanishing determinant.  In the case that the determinant does vanish, modify $\eta_1$ away from the tubular neighborhood of $L_2$ on which $\eta_2$ is supported by a small closed bump 2-form such that the new matrix has non-vanishing determinant.  Then define $\eta$ to be the form $\eta(u,v)$ solving this system.  

Notice that $[\eta]$ lives in the span of $[L_i]$ and contains no other classes.   The symplectic form produced in \cite{H} has the form
\[
\tilde \omega=\omega+t(\eta+\mbox{ exact forms})
\]
where $t$ is a small positive real number and thus $[\tilde\omega]=[\omega]+t[\eta]$.  Hence $[\tilde\omega]-l_1[L_1]-l_2[L_2]=[\omega]$ where $l_i=ta_i$.

\end{proof}

%

On $M=X\#_{F_g}Y$, this result will be used to reintroduce rim components into a symplectic form $\alpha_X+(c_X+c_Y)F+g\Gamma+\alpha_Y$.

A further method is inflation, which modifies a symplectic class $[\omega]$ using symplectic surfaces in $M$.
 
\begin{lemma}\label{l:infl} (\cite{H}, \cite{Bi}, \cite{LM}, \cite{M2}) 
Let $(M,\omega)$ be a symplectic 4-manifold and $V_1, V_2\subset M$ closed connected symplectic surfaces with $[V_i]^2\ge 0$ and which intersect transversely in a single positive point.  Then for every $\epsilon_i\ge 0$, the class
\[
[\omega]+\epsilon_1[V_1]+\epsilon_2[V_2]
\]
is represented by a symplectic form $\tilde \omega$ such that\begin{enumerate}
\item  any $\omega$-symplectic surface $Z$ meeting each $V_i$ non-negatively and transversely is $\tilde\omega$-symplectic for any choice of $\epsilon_i$ and 
\item $\omega$ and $\tilde \omega$ differ only on a neighborhood of $V_1\cup V_2$. 
\end{enumerate}
\end{lemma}

Inflation will be used in Theorem \ref{t:sym} to recover the rim components in a class $\alpha\in\mathcal P_M^F$.

The following theorem describes how to use these methods to produce a symplectic class with the correct rim components.  This follows a method used in \cite{H}.  The key idea is to start with a sum symplectic form $\omega$  in the class obtained by removing the rim classes from $\alpha$.  The form $\omega$ now needs to be modified to obtain a symplectic form $\tilde \omega$ that represents $\alpha$.  This is achieved by using Theorem \ref{t:lag} and Lemma \ref{l:infl} to reintroduce the rim components.  Note that both results modify the form $\omega$ only in a neighborhood of the representatives of $\mathcal R_i$ and $T_i$.

\begin{definition}\label{d:pfc}  Assume $M$ is an elliptic surface over a genus $g$ surface $\Sigma_g$ and $\omega$ a symplectic form on $M$ making $F_g$ symplectic. Assume there exists an open set $U\subset \Sigma_g$ such that
\begin{enumerate}
\item over $U$, $M$ is presented as a Lefschetz fibration with $F_g$ a fiber and at least 2 singular fibers whose vanishing cycles form a basis for $\pi_1(F_g)$ (matching vanishing cycles) and
\item $\omega$ is compatible with this fibration structure.
\end{enumerate}
Then $\omega$ is called partially fibration compatible (pfc) at $F_g$.  A class $\alpha\in H^2(M,\mathbb R)$ is pfc at $F_g$ if it has a pfc at $F_g$ representative. 

\end{definition}

The pfc property will be needed to construct certain representatives of the rim classes as will the following property.  

\begin{definition}Let $M$ be an elliptic surface, $F_g$ a smooth generic fiber and $X$,$Y$ elliptic surfaces such that $X\#_{F_g}Y=M$.  The sum $X\#_{F_g}Y$ is called a good sum for $M$ if either
\begin{enumerate}
\item the sum produces no rim components or
\item given any two pfc at $F_g$ classes in $X$ and $Y$, the diffeomorphism implicit in the fiber sum is chosen to glue the matching vanishing cycles from the X and Y side along their boundaries to generate two spheres.
\end{enumerate}

\end{definition} 

In the case of no rim components, this is just Def 3.7, \cite{DL}.  In the case of rim components, $M=E(n_1,g_1,p_1,...,p_k\#_{F_g}E(n_2,g_2,q_1,...,q_t)$ by Theorem \ref{t:class}.  Then it is always possible to choose the diffeomorphism such that this is a good sum, see Section 3.1, \cite{GS}.  In contrast, note that the manifolds $K(p_1,q_1;p_2,q_2;p_3,q_3)$ given in \cite{GM} are not good sums.

The following Theorem can be viewed as a generalization of Theorem \ref{t:38} to the case that rim-tori are present in the sum.

\begin{theorem}\label{t:sym} Let $M$ be an elliptic surface with $\chi(M)\ne 0$ and $F_g$ an oriented generic smooth fiber.  Let $\alpha\in\mathcal P^F_M$ be given as
\[
\alpha=e_1\mathcal R_1+d_1T_1+e_2\mathcal R_2+d_2T_2+\alpha_X+(c_X+c_Y)F+g(\Gamma_X+\Gamma_Y)+\alpha_Y.
\]
Assume $\alpha$ satisfies the following: \begin{enumerate}
\item $\alpha$ is sum balanced with respect to the good sum $M=X\#_{\tilde F_g}Y$.
\item  If at least one pair $(e_i,d_i)\ne (0,0)$, then each $\alpha_*+c_*F+g\Gamma_*$ is pfc over $\tilde F_g$.
\item  $\mathcal C^{\tilde F_g}_*={\mathcal C}_*^{F}$ with $*\in\{X,Y\}$.
\end{enumerate}
Then $\alpha\in\mathcal C_M^{F_g}$.
\end{theorem}

\begin{proof}  
 Denote the balanced rim pairs by 
 \begin{equation}\label{e:bal}
\alpha_{bal}=e_{1}\mathcal R_{1}+d_{1}T_{1}+e_{2}\mathcal R_{2}+d_{2}T_{2}.
\end{equation}

If $(e_1,d_1,e_2,d_2)=(0,0,0,0)$, then as $\alpha^2=\alpha_0^2>0$ and $\alpha$ is sum balanced, \cite{Go2} shows that  $\alpha\in \mathcal C_{M}^{F_g}$.

In fact, the same argument shows that $\alpha_0-\alpha_{bal}$ is representable by a symplectic form $\omega$ obtained from the symplectic sum.  

We may assume that at least one pair $(e_i,d_i)$ is non-zero, note that in this case $e_i\cdot d_i>0$.   Let $A\in\{\mathcal R_{i},T_{i}.  \}$  If  $\alpha\cdot A$ is $\left\{\begin{array}{c}\mbox{positive}\\\mbox{negative}\end{array}\right\}$, choose a representative for $\left\{\begin{array}{c}A\\-A\end{array}\right\}$ which is Lagrangian with respect to $\omega$.  (If $\alpha\cdot A=0$, then the corresponding rim pair can be ignored as this implies $(e_i,d_i)=(0,0)$.) For $\pm\mathcal R_{i}$ this can be done by the sum construction.  For $T_{i}$, note that $X$ and $Y$ have the structure of a Lefschetz fibration on an open set containing $\tilde F_g$ and the pfc condition ensures that the symplectic form $\omega$ can be chosen to be compatible with this fibration.  Moreover, in each case the open set conatins at least 2 singular fibers with matching vanishing cycles.  Then section 8, \cite{AMP}, ensures the existence of a Lagrangian sphere, produced via Lefschetz thimbles, in the class $\pm S_i$ which intersects $\pm\mathcal R_i$ transversally in a single point.    Apply Lagrangian surgery to $\pm S_{i}$ and $\pm\mathcal R_{i}$ (\cite{Po}; 2.2.1, \cite{DHL} summarizes the construction) to produce a Lagrangian $\pm T_i$.

Using Theorem \ref{t:lag}, we can ensure that the respective representatives can be made symplectic at the cost of a deformation of $\omega$ to a form $\omega'$ in the class
\[
\alpha_0+\alpha_F+\sum_{i=1}^2(\tilde e_{i}{\mathcal R}_{i}+\tilde d_{i} T_{i})
\]
where $\tilde e_{i}$ and $\tilde d_{i}$ have the correct sign, depending on the choice of $\pm(\mathcal R_i, T_i)$ made previously.

It follows from $\mathcal R_{i}\cdot T_{i}=1$ (and also $(-\mathcal R_i)\cdot (-T_i)=1$) that Lemma \ref{l:infl} is applicable to the pair $\pm(\mathcal R_{i}, T_{i})$ and it is possible to recover the coefficients $(e_{i}, t_{i})$ for $\alpha$.  Hence, $\alpha\in \mathcal C_{M}^{F_g}$.

\end{proof}

\section{Relative Symplectic Cone of $T^4\#l\overline{\mathbb CP^2}$}\label{s:50}

The goal of this section is to determine the relative symplectic cone for $T^4\#l\overline{\mathbb CP^2}$.  In this setting, there are no rim components.  Hence the arguments are somewhat simpler than for $\chi>0$ elliptic surfaces while still following the same general outline of those proofs:  The key issue will be in dealing with exceptional curves and the arguments will make use of automorphisms to ensure that a class is equivalent to a sum balanced class.

We first describe explicit automorphisms $T^4$ which will then be used to prove Theorem \ref{t:chi0}.

\subsection{Explicit Automorphisms of $T^4$}  

The 4-torus $T^4$ has intersection form $3H$.  Write any basis of $H^2(T^4,\mathbb Z)$ that represents this form as \[(F,\Gamma, A_1, A_2, B_1, B_2)\] and represent a class
\[
\alpha=cF+g\Gamma+a_1A_1+a_2A_2+b_1B_1+b_2B_2=(c,g,a_1,a_2,b_1,b_2)\in H^2(T^4,\mathbb R).
\] 
In \cite{N}, the geometric automorphism group for $T^2\times \Sigma_g$, $g\ge 2$, is described.  While this result does not apply to $T^4$, certain diffeomorphisms are defined for $T^4$.  In particular, the maps denoted by $R_*$ in \cite{N} lead to automorphisms of $2H$-type as in Lemma \ref{l:rp}.  These maps are generated by a Dehn twist along a generator of $H_1(T^2,\mathbb Z)$ in the base torus and the identity map on the fiber torus.  This map is thus non-trivial only on $T^2\times S^1\times (-\epsilon,\epsilon)$.  On cohomology the induced maps are

\[\left(\begin{matrix}
c \\
g \\
a_1 \\
a_2 \\
b_1 \\
b_2
\end{matrix}\right)\stackrel{I}{\mapsto}
\left(\begin{matrix}
c \\
g \\
a_1 \\
a_2-Nb_1 \\
b_1 \\
b_2+Na_1
\end{matrix}\right)
\mbox{  and  }
\left(\begin{matrix}
c \\
g \\
a_1 \\
a_2 \\
b_1 \\
b_2
\end{matrix}\right)\stackrel{II}{\mapsto}
\left(\begin{matrix}
c \\
g \\
a_1-Nb_2 \\
a_2 \\
b_1+Na_2 \\
b_2
\end{matrix}\right).
\]

A further source for automorphisms is the explicit geometric description of the torus as a quotient.   Let $T^2\times T^2=\mathbb R^4/ \mathbb Z^4$ with coordinates $t_1,..,t_4$.  Choose the projection $(t_1,t_2,t_3,t_4)\mapsto (t_3,t_4)$ to define the bundle.  Let the generating classes for $H^2(T^4,\mathbb R)$ be given as 
\[
\begin{array}{cc}
F=dt_1\wedge dt_2,&\Gamma=dt_3\wedge dt_4\\
B_1=dt_1\wedge dt_3,&B_2=dt_4\wedge dt_2\\
A_1=dt_1\wedge dt_4,&A_2=dt_2\wedge dt_3\\
\end{array}
\]
where the classes are the same as previously.  Any $T\in SL(4,\mathbb Z)$ defines a diffeomorphism of $T^4$.  The following list, together with the induced action on cohomology, will be useful:
\begin{enumerate}
\item The interchange map from the $\chi>0$ case can also be obtained:
\[
T=\left(\begin{matrix}
1 & 0 & 0 & 0 \\
0 & 1 & 0 & 0 \\
0 & 0 & 0 & 1 \\
0 & 0 & -1 & 0
\end{matrix}\right)
\mbox{  inducing  }
\left(\begin{matrix}
c \\
g \\
a_1 \\
a_2 \\
b_1 \\
b_2
\end{matrix}\right)\mapsto
\left(\begin{matrix}
c\\
g \\
b_1\\
b_2 \\
-a_1\\
-a_2
\end{matrix}\right)
\]
\item  Let $A\in\mathbb Z$.  Then
\[
T=\left(\begin{matrix}
1 & 0 & A & 0 \\
0 & 1 & 0 & 0 \\
0 & 0 & 0 & 1 \\
0 & 0 & -1 & 0
\end{matrix}\right)\]
followed by the interchange map induces the map
\[
\left(\begin{matrix}
c \\
g \\
a_1 \\
a_2 \\
b_1 \\
b_2
\end{matrix}\right)\stackrel{III}{\mapsto}
\left(\begin{matrix}
c \\
g-a_1A \\
 a_1\\
a_2+cA\\
b_1\\
b_2
\end{matrix}\right)
\]
Note that this map changes the fiber class from $F$ to $F-AA_2$.  The right hand vector is written with respect to the new basis.

\item Let $A\in\mathbb Z$.  Then
\[
T=\left(\begin{matrix}
1 & 0 & 0 & 0 \\
0 & 1 & 0 & A \\
0 & 0 & 0 & 1 \\
0 & 0 & -1 & 0
\end{matrix}\right)\]
followed by the interchange map induces the map
\[
\left(\begin{matrix}
c \\
g \\
a_1 \\
a_2 \\
b_1 \\
b_2
\end{matrix}\right)\stackrel{IV}{\mapsto}
\left(\begin{matrix}
c \\
g+Aa_2\\
a_1-Ac\\
a_2\\
b_1\\
b_2
\end{matrix}\right)
\]
Note that this map changes the fiber class from $F$ to $F+AA_1$.  The right hand vector is written with respect to the new basis.

%

\end{enumerate}

\begin{lemma}\label{l:t4} Let $\alpha_0\in\mathcal P^F_{T^4}$ and assume $(a_1,b_2)$ are not a multiple of an integral class.  Let $\epsilon>0$.  Then there exists an automorphism of $H^2(T^4,\mathbb Z)$ which covers a self-diffeomorphism of $T^4$ and sends $\alpha_0$ to $\alpha\in\mathcal P^{\tilde F}_{T^4}$ with 
\[
\alpha\cdot \tilde F<\epsilon
\] while possibly changing the fiber class.
\end{lemma}

\begin{proof}
Assume first that $(a_1,b_2)$ are linearly independent over $\mathbb Z$.  For a given class $\alpha_0$ represented by $(c,g,a_1,a_2,b_1,b_2)$, use map I to produce a class with coefficients  $(c,g,a^1_1,a^1_2,b^1_1,b^1_2)$ with 
\[
0<|b^1_2|\le\frac{|a_1|}{2}.
\]
Now apply map II to the class $(c,g,a^1_1,a^1_2,b^1_1,b^1_2)$ to produce a new class $(c,g,a^2_1,a^2_2,b^2_1,b^2_2)$ with
\[
0<|a^2_1|\le \frac{|b^1_2|}{2}\le \frac{|a_1|}{4}.
\]
Iterate this procedure until the terms $a^k_1$ and $b^k_2$ satisfy
\[
0<|a^k_1|, |b_2^k|\le \epsilon.
\]
Up to this point, the terms $(c,g)$ have not been changed, the class has the form
\[
(c,g,\underbrace{\tilde a_1}_{<\epsilon},\tilde a_2,\tilde b_1,\underbrace{\tilde b_2}_{<\epsilon})
\]  Now apply map III to reduce the $g$ coefficient to satisfy $0<g\le \epsilon$.  At this point the fiber class has been changed and a class, using the same ordering, of the form
\[
(c,\underbrace{ g^1}_{<\epsilon},\underbrace{\tilde a^1_1}_{<\epsilon},\tilde a^1_2,\tilde b^1_1,\underbrace{\tilde b^1_2}_{<\epsilon})
\]
obtained.  

%
%

\end{proof}

\subsection{The Relative Symplectic Cone}

The relative symplectic cone for $T^4$ was determined by Geiges:

\begin{theorem}(Theorem 2, \cite{Ge})\label{t:old} Let $T$ be an orientable $T^2$-bundle over $T^2$ with generic smooth fiber $F_g$.  Assume $\mathcal C_T^{F_g}\ne \emptyset$.  Then 
\[
\mathcal C^{F_g}_T=\mathcal P^F_T.
\]
Moreover, each class $\alpha\in \mathcal C^{F_g}_T$ can be represented by a symplectic form compatible with the fibration.
\end{theorem}

In fact, if $T$ is minimal, then $\mathcal C_T=\mathcal P_T$, see \cite{Ge}.

Recall that for $T^4$ we have fixed, for convenience, a specific fibration.  The previous Theorem does not and the result is valid for any fibration structure placed on the total space $T$ with fiber $F_g$.  It is then rather straightforward to construct explicit examples of symplectic forms compatible with the fibration, hence $\mathcal C_{T^4}^{F_g}\ne \emptyset$
.

The main result of this section can now be stated:

\begin{theorem}\label{t:chi0}
Let \[
M=T^4\#l\overline{\mathbb CP^2}
\]
 with $l\ge 1$, and $F_g$ a generic oriented fiber.  Let $K\in\mathcal K_F$ and denote the $l$ exceptional curves in $\mathcal E_K$ by $E_i$.  Then
 \[
 \mathcal C_{M,K}^{F}\subset \mathcal C_M^{F_g}.
 \] 
 In particular, 
 \[
 \mathcal C_{M}^{F_g}=\{\alpha\in\mathcal P^F_M\;|\; \alpha\cdot E_i\ne 0\;\forall i=1,..,l\}.
 \] 
\end{theorem}

\begin{proof} Fix a symplectic canonical class $K\in\mathcal K_F$.  Let $\alpha_0\in\mathcal C_{M,K}^{F}$ such that $\alpha_0\cdot E_i>0$ for all $i\in\{1,..,l\}$.  Write the class as
\[
\alpha_0=\alpha_{min}-\sum_{i=1}^le_iE_i
\]

with $e_i>0$.  Note that $\alpha^2_{min}>0$ and $\alpha_{min}\cdot F>0$.  Hence $\alpha_{min}\in\mathcal C_{T^2\times T^2}^{F_g}$.

Consider first the case that in $\alpha_{min}$ the pair $(a_1,b_2)$ is not a multiple of an integral class.  Given $\epsilon>0$, this class can be mapped to an equivalent class
\[
\alpha=\alpha_{1}+\alpha_2+c\tilde F+g\Gamma
\]
by a diffeomorphism of $T^2\times T^2$ (Lemma \ref{l:t4}).  This class satisfies $0< g <\epsilon$.

Note that the fiber class may have changed under this diffeomorphism and the symplectic canonical class changed to $\tilde K$.  Theorem \ref{t:s2t2} now implies that $\alpha_0\in\mathcal C_M^{F_g}$.

Assume now that in in $\alpha_{min}$ the pair $(a_1,b_2)$ is a multiple of an integral class.  Choose a $\delta>0$ such that for the class 
\[
\alpha_\delta=\alpha_{min}-\delta F
\]
\begin{enumerate}
\item the class obtained from $\alpha_{\delta}$ by the map IV with $A=-1$ has coefficients $(\tilde a_1,\tilde b_2)$ which are not a multiple of an integral class and
\item $\alpha_\delta^2>0$.

\end{enumerate}  
The previous argument then shows that $\alpha_\delta\in\mathcal C^{F_g}_M$ and $F_g$ is $\alpha_\delta$-symplectic.  Now inflate along $F$ to regain the original class $\alpha_0$, hence $\alpha_0\in\mathcal C_M^{F_g}$.

\end{proof}

This result can be used to extend Theorem 4.1, \cite{DL}, using Theorem \ref{t:exc}.

\begin{cor}\label{c:chi0}
Let \[
M=(T^2\times\Sigma_g)\#l\overline{\mathbb CP^2}
\]
 with $l\ge 0$, and $F_g$ a generic oriented fiber.  Let $K\in\mathcal K_F$ and denote the $l$ exceptional curves in $\mathcal E_K$ by $E_i$.  Then
 \[
 \mathcal C_{M,K}^{F}\subset \mathcal C_M^{F_g}.
 \] 
 In particular, 
 \[
 \mathcal C_{M}^{F_g}=\{\alpha\in\mathcal P^F_M\;|\; \alpha\cdot E_i\ne 0\;\forall i=1,..,l\}.
 \] 
\end{cor}

\begin{proof}
The case $l=0$ is exactly Theorem 4.1, \cite{DL}.  For $g=1$, this is Theorem \ref{t:chi0}.  Hence assume that $g\ge 2$ and $l\ge 1$.  Then write
\[
(T^2\times\Sigma_g)\#l\overline{\mathbb CP^2}=(T^2\times\Sigma_{g-1})\#\left(T^4\#l\overline{\mathbb CP^2}\right)
\]
and apply Theorem \ref{t:exc}.
\end{proof}

\subsection{The Symplectic Cone}\label{s:t4non}

The relative result for $T^4\#l\overline{\mathbb CP^2}$ allows us to determine the full symplectic cone. The result has been obtained in \cite{EV} and \cite{LMS} previously, we give a brief alternative proof of this case.

\begin{lemma} 
\[
\mathcal C_{T^4\#l\overline{CP^2}}=\{\alpha\in\mathcal P_{T^4\#l\overline{CP^2}}\;|\;\alpha\cdot E_i\ne 0\;\forall i\in\{1,..,l\}\}
\]
\end{lemma}

\begin{proof} Clearly the inclusion $\subset$ holds.

Let $\alpha\in \{\alpha\in\mathcal P_{T^4\#l\overline{CP^2}}\;|\;\alpha\cdot E_i\ne 0\;\forall i\in\{1,..,l\}\}$ be the class $(c,g,a_1,a_2,b_1,b_2,e_1,...,e_l)$.  Then $(c,a_1,b_1)$ is not the zero vector, hence $\alpha$ lies in one of the relative symplectic cones $\mathcal C_{T^4\#l\overline{CP^2}}^{\pm F_g}$, $\mathcal C_{T^4\#l\overline{CP^2}}^{\pm A_1}$ or $\mathcal C_{T^4\#l\overline{CP^2}}^{\pm B_1}$. 

\end{proof}

\section{Relative Symplectic Cones for Elliptic surfaces with $\chi>0$}  \label{s:cp}

In order to apply Theorem \ref{t:sym} to determine the relative symplectic cones of elliptic surfaces with positive Euler number, we need to show that every balanced class is equivalent to a sum balanced class and then show that every symplectic class is pfc with regard to some smooth fiber.  Lemma \ref{l:ba} ensures that every class in $\mathcal P^F_M$ is equivalent to a balanced class.  Complications arise in sums which involve $E(1)$, due to the presence of exceptional curves which are not constrained to a fiber and the more complicated structure of the symplectic cone arising in the $b^+=1$ case.

For this reason, special attention is given to sums $M=E(1)\#_{F_g} N$.  These also form the basis for an inductive argument:  After recalling results on the relative symplectic cone of $E(1)$, use this and Theorem \ref{t:sym} to determine the relative symplectic cones of $E(2)$ and $E(2)\#l\overline{\mathbb CP^2}$.  Using a result in \cite{FM1}, it follows that every class in the relative symplectic cone of $E(1)$ is pfc relative to any smooth fiber, this extends to the relative symplectic cones of $E(2)$ and $E(2)\#l\overline{\mathbb CP^2}$ in a suitable way.  Iterate this procedure to obtain the relative symplectic cones of $E(n)\#l\overline{\mathbb CP^2}$.

Throughout we assume  that the gluing diffeomorphism has been chosen such that the sum is good.

\subsection{$E(1)$: Basic Results}

We briefly review the structure of the symplectic cone for $E(1)$.  For a symplectic canonical class $K$ on the underlying smooth manifold of $E(1)$ there always exists a basis of the second (co)homology which consists of the proper transform of the generator of the second (co)homology of $CP^2$ and $9$ pairwise orthogonal exceptional classes.  Call such a basis a $K$-standard basis and write it as $(H^K, E_1^K,...,E_9^K)$.  When there is no confusion as to what the symplectic canonical class is, we often drop the superscript.  Recall that \[
H'=\left(\begin{matrix}
0&1\\1&-1
\end{matrix}\right).
\] 
The following results will be useful:

\begin{lemma}\label{l:e1setup} Let $M=E(1)$. \begin{enumerate}  
\item (Lemma 3.5, \cite{TJLL})	Given two symplectic canonical forms $K_1$ and $K_2$, there exists a diffeomorphism of $M$ which maps the $K_1$-standard basis to the $K_2$-standard basis.
\item (Prop 4.9, \cite{TJLL})	A class $\alpha_K$ is represented by a $K$-symplectic form if and only if it is equivalent to a reduced class with respect to the $K$-standard basis such that no coefficient vanishes. This means $\alpha_K$ can be written as $aH^K-\sum b_iE_i^K$ with 
\[
a\ge b_1+b_2+b_3
\]
and
\[
b_1\ge b_2\ge...\ge b_9>0.
\]
\item Given $K$ and an exceptional class $E$, it is possible to find a $K$-standard basis $(H^K, E_1^K,...,E_8^K, E)$ such that the intersection form splits into $E_8$ and $H'$, where $H'$ is generated by $-K$ and $E$.  Call an $E_8$ generated in such a splitting for such a basis a $K$-standard generated $E_8$.
\item (Prop1.2.12, \cite{MSch}; Prop 2.7, \cite{TJLW}) 	Given a splitting of the intersection form of $M$ into $E_8\oplus H'$ where $H'$ is generated by $-K$ and $E$, there is a diffeomorphism of $M$ that takes this splitting to a splitting $E_8\oplus H'$ with a $K$-standard generated $E_8$ and leaving $H'$ unchanged.  Moreover, this diffeomoprhism is generated by reflections on $-2$-spheres disjoint from $-K$ and $E$.

\end{enumerate}
\end{lemma}

Note that the last diffeomorphism of $E(1)$ extends to a diffeomorphism of  $E(n)$ affecting only an $E_8$-component and otherwise acting by identity.

These results together with Theorem \ref{t:b10} imply the following result.

\begin{theorem}\label{t:be1} Let $M=E(1)\#l\overline{\mathbb CP^2}$, $l\ge 0$, and $F_g$ a generic oriented smooth fiber of $E(1)$ disjoint from the blow-up locus.  Then
\[
\bigsqcup_{K\in\mathcal K(F_g)}\left\{\left.\alpha\in \mathcal P_{M}^{F}\;\right|\; \alpha\cdot E>0\; \forall E\in \mathcal E_{K}\right\}= \mathcal C^{F_g}_{M}.
\]
In particular, $\mathcal C_{E(1)}^{F_g}=\left\{\left.\alpha\in \mathcal P_{M}^{F}\;\right|\; \alpha\cdot E>0\; \forall E\in \mathcal E_{-F}\right\}=\mathcal C_{E(1)}^F$. 
\end{theorem}

The following result ensures the classes in $\mathcal C_{E(1)\#l\overline{\mathbb CP^2}}^{F_g}$ can be used in Theorem \ref{t:sym}.

\begin{lemma}\label{l:pfc}(Lemma 2.13, Prop 3.4, \cite{FM1})  Let $\alpha\in \mathcal C_{E(1)\#l\overline{\mathbb CP^2}}^{F}$.  Then $\alpha$ is pfc with respect to any smooth fiber $F_g$ not containing an exceptional curve.

\end{lemma}

The splitting $M=E(1)\#_FN$ determines on $E(1)$ a symplectic canonical class $K=-F$. There are now two basis in which to study the $E(1)$ classes, the standard basis of Lemma \ref{l:e1setup} and the $E_8\oplus H$ basis that is more naturally associated to the elliptic surface $M$.  Lemma \ref{l:e1setup} allows us to choose a $K$-standard basis and the splitting $E_8\oplus H$ may be assumed to have a $K$-standard generated $E_8$.  Thus we may write the class $\alpha_{E(1)}$ in the following two ways:
\begin{equation}\label{e:e1}
\alpha_{E(1)}=aH^K-\sum_{i=1}^8 b_iE_i^K-(c-g)E=\sum_{i=0}^7k_i D_i+gE+cF
\end{equation}
where 
\[
\{D_i\}_{i=0}^7=\{H^K-E_1^K-E_2^K-E_3^K, E_1^K-E_2^K, E_2^K-E_3^K,....,E_7^K-E_8^K\}
\]
which has $E_8$ as its intersection form and $E$ is a an exceptional class which is a section of $E(1)$.  Call the first the standard form, the second the split form.  For convenience, we will write
\[
\alpha_{E(1)}=(a,b_1,..,b_8,c-g)=(k_0,..,k_7,g,c).
\]
Given the vector $(k_0,..,k_7,g,c)$, the base change is explicitly given by
\begin{equation}\label{e:bc}
\left(\begin{matrix}
k_0 \\
k_1 \\
k_2 \\
k_3 \\
k_4 \\
k_5 \\
k_6 \\
k_7 \\
g \\
c
\end{matrix}\right)\mapsto
\left(\begin{matrix}
a \\
b_1 \\
b_2 \\
b_3 \\
b_4 \\
b_5 \\
b_6 \\
b_7 \\
b_8 \\
c-g
\end{matrix}\right)=
\left(\begin{matrix}
3c+k_0 \\
c+k_0-k_1 \\
c+k_0+k_1-k_2 \\
c+k_0+k_2-k_3 \\
c+k_3-k_4 \\
c+k_4-k_5 \\
c+k_5-k_6 \\
c+k_6-k_7 \\
c+k_7 \\
c-g
\end{matrix}\right).
\end{equation}

Both viewpoints will be used in Theorem \ref{t:e1}, which finally determines the relative symplectic cone for sums with $E(1)$.

\subsection{$E(2)$: Balanced Classes}

Lemma \ref{l:ba} did not include $E(2)$.  This section focuses on balancing classes in $E(2)$.

If $M$ is diffeomorphic to $E(2)$, then there are not enough $2H$-terms to shift the volume off of the rim pairs involved in the sum.  Further, as will be seen in the proof of Theorem \ref{t:e1}, in order for the class to be sum balanced (see Def. \ref{d:sbal}), a condition will be imposed on the relative sizes of $\alpha_0\cdot F$ and $\alpha_0\cdot \Gamma$ which arises from the symplectic cone of $E(1)$.  Both of these issues are dealt with in the following.

For those diffeomorphic to $E(2)$, note that by Theorem \ref{t:lo}, the image of $Diff^+(M)$ in $O$ is $O'$.  Hence the fiber class no longer needs to be preserved and any map of spinor norm one can be used.  In particular, the section $\Gamma$ is a sphere of self-intersection $-2$ and this class can be used to generate a diffeomorphism.  Further, considering instead of the pair $(F,\Gamma)$ the pair $(F, W=\Gamma+F)$, we obtain a new pair that behaves like a rim pair (and has intersection matrix $H$):\begin{enumerate}
\item The class $cF+g\Gamma$ becomes $(c-g)F+gW$.
\item Reflection on $\Gamma$ maps $(c-g)F+gW$ to $gF+(c-g)W$.
\item Lemma \ref{l:rp} can be applied to the pairs $(F,W)$ and $(\mathcal R, T)$.
\end{enumerate}
However, these maps come at the cost of changing the fiber class, hence changing the initial fibration to a new, albeit diffeomorphic, one.

In the following proof, the actual basis elements will not be relevant, only keeping track of the coefficients and which ones correspond to the fiber, the section and the rim-pairs will matter.  Hence, it will be convenient to continue to use the vector notation, where the notation in the vector tracks the location of the fiber, the section, and the $H$-pair even though the two vectors are with respect to two distinct basis.  Hence, the map in Lemma \ref{l:e2} will be written as

\begin{equation}\label{e:2}
\left(\begin{matrix}
w \\
g \\
a \\
b
\end{matrix}\right)
\mapsto
\left(\begin{matrix}
w \\
g-ib \\
a+iw \\
b
\end{matrix}\right)
\end{equation}

%
%
%

\begin{theorem}\label{t:e2}
Let $M$ be diffeomorphic to $E(2)$ with a given fibration having $F_g$ as a generic fiber.  With respect to this fibration, let $\alpha_0\in \mathcal P_M^{F}$ and let $\epsilon>0$.  Then there exists a self-diffeomorphism of $M$ which sends the given fibration to one with generic fiber $\tilde F_g$ and $\alpha_0$ to 
\[
\tilde\alpha_0=\sum_{i=0}^7k_{i,1}D_{i,1}+\sum_{i=0}^7k_{i,2}D_{i,2}+a_{1}\mathcal R_{1}+b_{1}T_{1}+a_{2}\mathcal R_{2}+b_{2}T_{2}+w\tilde F+gW.
\]
 such that
\begin{enumerate}
\item $\tilde\alpha_0\in \mathcal P_M^{\tilde F}$,
\item $\tilde\alpha_0$ is balanced with $0\le a_i\cdot b_i\le \epsilon$,
\item  $0\le |k_{i,j}|<\epsilon$ and 
\item \begin{enumerate}
\item either $0<g<\epsilon$ or
\item $\tilde\alpha_0=w\tilde F+g W$ is a multiple of an integral class.
\end{enumerate}
\end{enumerate}
\begin{proof}
 The intersection form of $M$ with respect to $F$ is given by $2E_8\oplus 2H\oplus \langle F, \Gamma\rangle$.  Write 
 \[
 \alpha_0=\alpha^0_{8,1}+\alpha^0_{8,2}+a^0_{1}\mathcal R_{1}+b^0_{1}T_{1}+a^0_{2}\mathcal R_{2}+b^0_{2}T_{2}+w^0F+g^0W.
 \]
 
 Apply Theorem \ref{t:mine1} first to 
 \[
 \alpha^0_{8,1}+a^0_{1}\mathcal R_{1}+b^0_{1}T_{1}+a^0_{2}\mathcal R_{2}+b^0_{2}T_{2}+w^0F+g^0W\]
  and then to the newly obtained class using the other $E_8$-term,
  \[
 \alpha^0_{8,2}+\tilde a^0_{1}\mathcal R_{1}+\tilde b^0_{1}T_{1}+\tilde a^0_{2}\mathcal R_{2}+\tilde b^0_{2}T_{2}+\tilde w^0F+g^0W,
 \]
  to obtain an equivalent class
  \[
  \alpha_1=\alpha^1_{8,1}+\alpha^1_{8,2}+a^1_{1}\mathcal R_{1}+b^1_{1}T_{1}+a^1_{2}\mathcal R_{2}+b^1_{2}T_{2}+w^1F+gW.
\]

This leads to the following possible configurations in $\alpha_1$:

\[
\begin{array}{|c|c|c|c|c|}
\hline
\alpha^0_{8,1} & \alpha_{8,2}^0 & (a^0_{1},b^0_{1},a^0_{2},b^0_{2}) & &\\\hline
\downarrow&\downarrow&\downarrow&&\\\hline
<\epsilon&0&(0,0,a,b)&\rightarrow &\mbox{ Case 0}\\\hline
<\epsilon&<\epsilon&<\epsilon&\rightarrow &\mbox{ Case 1}\\\hline
0&<\epsilon&<\epsilon&\rightarrow &\mbox{ Case 1}\\\hline
0&0&(0,0,a,b)&\rightarrow &\mbox{ Case 2}\\\hline
\end{array}
\]
 Note that in each case the condition on the $E_8$-coefficients in satisfied.

\noindent{\bf Case 0:}  First, assume that the class $(\alpha^1_{8,1}, 0,0,a,b)$ is a multiple of an integer class.  Then by Lemma \ref{l:ham} it is equivalent to a class $(0,0,0,\tilde a, \tilde b)$ of the same square and divisibility.

Assume now that the class $(\alpha^1_{8,1}, 0,0,a,b)$ is not a multiple of an integer class.  Hence applying Theorem \ref{t:mine1} to this part of $\alpha_1$ will result in the $\alpha_{8,1}^1$ and rim parts being minimized.  

Thus a class $\alpha_2$, equivalent to $\alpha_0$, is obtained which has the following behavior:

\[
\begin{array}{|c|c|c|c|c|}
\hline
\alpha^0_{8,1} & \alpha_{8,2}^0 & (a^0_{1},b^0_{1},a^0_{2},b^0_{2}) & &\\\hline
\downarrow&\downarrow&\downarrow&&\\\hline
<\epsilon&0&<\epsilon&\rightarrow &\mbox{ Case 1}\\\hline
0&0&(0,0,\tilde a,\tilde b)&\rightarrow &\mbox{ Case 2}\\\hline
\end{array}
\]

\noindent{\bf Case 1:}  Consider for simplicity of notation the class \[
\alpha=\alpha_{8,1}+\alpha_{8,2}+a_{1}\mathcal R_{1}+b_{1}T_{1}+a_{2}\mathcal R_{2}+b_{2}T_{2}+wF+gW
\]
where each entry in $\alpha_{8,1}+\alpha_{8,2}$ has magnitude bounded by $\epsilon$  and $0< 2a_i\cdot b_i<\epsilon$.

If one of the $(\mathcal R_i, T_i)$ coefficients is non-vanishing, then Case 1 of the proof of Theorem \ref{t:mine1} shows that at least one coefficient is non-vanishing and bounded by $\epsilon$.  Assume this coefficient is $b_1$.  Moreover, Case 1 also implies that $0\le |a_2|,|b_2|<\epsilon$.  Now use $b_1$ and the map in \eqref{e:2} to obtain a class $\tilde\alpha$ with $\tilde g=g-ib_1>0$ and $2w\tilde g$ smaller than $\epsilon$:
\[
\tilde\alpha=\underbrace{\alpha_{8,1}}_{0\le |k_i|<\epsilon}+\underbrace{\alpha_{8,2}}_{0\le |k_i|<\epsilon}+(a_{1}+iw)\mathcal R_{1}+\underbrace{b_{1}}_{<\epsilon}\tilde T_{1}+\underbrace{a_{2}}_{<\epsilon}\mathcal R_{2}+\underbrace{b_{2}}_{<\epsilon}T_{2}+w\tilde F+\underbrace{\tilde g}_{<\epsilon}W.
\]
Note that $\tilde F=F-iT_1$.

  If $\epsilon$ is chosen $<<\alpha_0^2$ and so that $w>\tilde g$, then this ensures that $\tilde\alpha-w\tilde F-\tilde gW$ has positive square and $w>\tilde g$.    Note that this forces $(a_1+iw)\cdot b_1>\alpha_0^2>0$.  As the fiber class has changed, at this point a diffeomorphism may be applied to the $E_8$ components as needed (see for example Lemma \ref{l:e1setup}.4).  While this preserves the small square of the $E_8$-terms, it may vary the size of individual terms.  Hence use the pair $(a_1+iw,b_1)$ in Lemma \ref{l:e8h} to shrink these again while changing $a_1+iw$ to $\mathfrak a$ and preserving $b_1$, thus producing a class equivalent to $\tilde \alpha$:
  \[
  \underbrace{\tilde\alpha_{8,1}}_{0\le |k_i|<\epsilon}+\underbrace{\tilde\alpha_{8,2}}_{0\le |k_i|<\epsilon}+\mathfrak a\mathcal R_{1}+\underbrace{b_{1}}_{<\epsilon}\tilde T_{1}+\underbrace{a_{2}}_{<\epsilon}\mathcal R_{2}+\underbrace{b_{2}}_{<\epsilon}T_{2}+w\tilde F+\underbrace{\tilde g}_{<\epsilon}W.
  \]

  The assumption on $\epsilon$ again ensures that after the $E_8$-terms have been made small, we still have $\mathfrak a \cdot b>0$.  Finally, use Lemma \ref{l:fi} to shrink the $\mathfrak a$ term to be smaller than $\epsilon$ while increasing the $ w$-term and preserving both $\tilde g$ and $b$.  The class
\[
  \underbrace{\tilde\alpha_{8,1}}_{0\le |k_i|<\epsilon}+\underbrace{\tilde\alpha_{8,2}}_{0\le |k_i|<\epsilon}+\underbrace{ \mathfrak a_1}_{<\epsilon}\mathcal R_{1}+\underbrace{b_{1}}_{<\epsilon}\tilde T_{1}+\underbrace{a_{2}}_{<\epsilon}\mathcal R_{2}+\underbrace{b_{2}}_{<\epsilon}T_{2}+\tilde w\tilde F+\underbrace{\tilde g}_{<\epsilon}W
\]
is equivalent to $\alpha_0$ and now satisfies the claim.
 
If all of the $(\mathcal R_i, T_i)$ coefficients are 0, then use Lemma \ref{l:e8h} applied to $(\alpha_{8,i},0,0)$ for $i=1$ or $i=2$, noting that each term in $\alpha_{8,i}$ is bounded by $\epsilon$, to generate a coefficient for a rim pair that is as needed for the argument above to work.  If all the terms in each $\alpha_{8,i}$ vanish, then $\alpha_0$ is equivalent to a class $wF+gW$. 

Assume that  $\alpha_0$ is equivalent to $w\tilde F+g\tilde W$ and $\frac{w}{g}\not\in \mathbb Q$, then apply Lemma \ref{l:bal} to the class $w\tilde F+g\tilde W+0\mathcal R_1+0T_1$ to obtain an equivalent class 
\[
\tilde w F_1+\tilde g W_1+\tilde a\tilde{\mathcal R_1}+\tilde b\tilde T_1
\]
with $0<\tilde g, |\tilde a_1|, |\tilde b_1|<\epsilon$.  

\noindent{\bf Case 2:}  Assume now that $\alpha_0$ is equivalent to a class $a\mathcal R+bT+wF+gW$.  This case can only occur if $\alpha_0-wF-gW$ is a multiple of an integral class.  Using Lemma \ref{l:fi} or \ref{l:e2}, the coefficient of $\mathcal R$ in $\alpha_0$ can be changed by some multiple of $w$ or $g$.  If now $\alpha_0-wF-gW$ is no longer a multiple of an integral class, then apply the procedure in Case 1 and the claim follows.

This leaves the case that $\alpha_0$ itself is a multiple of an integral class.  Assume that it is integral (i.e. ignore the multiplying factor).  Lemma \ref{l:ham} implies that $\alpha_0$ is equivalent to a class $\mathcal R+bT+wF+gW$ with $b,w,g\in\mathbb Z$ and $w,g>0$.  Then map this class as follows:
\[
\left(\begin{matrix}
w \\
g \\
1 \\
b
\end{matrix}\right)\stackrel{-2-\mbox{reflection}}{\mapsto}
\left(\begin{matrix}
w \\
g \\
b \\
1
\end{matrix}\right)
\stackrel{Eq.\;\ref{e:2}: i=-g+1}{\mapsto}
\left(\begin{matrix}
w \\
1 \\
b+(-g+1)w \\
1
\end{matrix}\right)\mapsto
\]
\[
\stackrel{Lemma\;\ref{l:fi}: i=-b+(g-1)w}{\mapsto}
\left(\begin{matrix}
w+b+(-g+1)w \\
1 \\
0\\
1
\end{matrix}\right)\mapsto 
\left(\begin{matrix}
w+b+(-g+1)w \\
1 \\
0\\
0
\end{matrix}\right)
\]
Note that $w+b+(-g+1)w\ge 1$.  

\end{proof}

\end{theorem}

\subsection{$E(1)$ Sums: Sum Balanced Classes and Relative Symplectic Cones}
Consider first manifolds of the form $E(1)\#_{F_g} N$, where it is not avoidable to have the $E(1)$-term.  The goal is to show how to obtain a class $\alpha_{E(1)}\in\mathcal C^F_{E(1)}$ in the splitting of $\alpha-\alpha_{bal}=\alpha_{E(1)}+\alpha_N$.  This will allow us to determine the relative symplectic cone of $E(2)\#l\overline{\mathbb CP^2}$.   

Two situations are distinguished, sums with finitely many exceptional curves and sums with $E(1)\#\overline{\mathbb CP^2}$, where the presence of additional exceptional curves, which also appear in the relative symplectic cone in Theorem \ref{t:be1}, complicates matters.

\begin{theorem} \label{t:e1} Let $M=E(1)\#_{\tilde F_g}N_i$ be an elliptic surface and $F_g$ a generic oriented fiber.  Fix a symplectic canonical class $K\in\mathcal K_F$ on $M$ and let $\mathcal E_K=\{E_1,...,E_l\}$.  Let $\alpha_0\in\mathcal C_{M,K}^{F}$.   Assume $N_i$ is one of the following:\begin{itemize}
\item  An elliptic surface $N_1$ 
\begin{itemize}
\item with exactly $l$ exceptional spheres $E_i$,
\item $\mathcal C^{F}_{N_1,K}\subset \mathcal C_{N_1}^{F_g}$,
\item $E(1)\#_{\tilde F_g}N_i$ generates no rim-tori in the sum or each $\alpha \in \mathcal C^{F}_{N_1,K}$ is pfc at $\tilde F_g$ and 
\item such that the intersection form of $N_1$ contains $2H$.
\end{itemize} 
\item $N_2=E(1)\# l\overline{\mathbb CP^2}$ with $l\ge 0$.  If $l>0$, assume further that each $E_i$ is a $K$-exceptional class in $N_2$.  
\end{itemize}
Then $\alpha_0$ is equivalent to a sum balanced class $\alpha\in\mathcal C_{M,K}^{F}$.  Therefore, 
\[
\mathcal C_{M,K}^{F}\subset \mathcal C_{M}^{F_g}.
\]

\end{theorem}

\begin{proof} The proof will be broken into three cases.  In the first case the key is to show how the sum balance can be achieved on the $E(1)$ side.  In the second case $N=E(1)$, hence the argument from the first case will need to be applied to both sides.  Finally, in the third case $N=E(1)\# l\overline{\mathbb CP^2}$ with $l>0$, the additional exceptional spheres will need to be accounted for.

 {\bf I:} Assume first that $N=N_1$. Given $\epsilon>0$ and using the assumption on $N_1$ that its intersection form contains $2H$, Cor \ref{c:bl} or Lemma \ref{l:ba} provide a class $\alpha=\alpha_{E(1)}+\alpha_N+\alpha_{bal}$ where \begin{enumerate}
\item $\alpha_{E(1)}=\alpha_8+c_1F+gE$ and the $E_8$-component $\alpha_{8}$ (given in the form \eqref{e:e1}) has coefficients $k_i$ which satisfy $|k_i|< \epsilon$, $k_0\le 0$ and $k_i\ge 0$ for $i\ge 1$;
\item $\alpha_N$ contains the $E_i$-contributions of $\alpha$; 
\item  $\alpha_{bal}$ is given by \eqref{e:bal} and contains the balanced rim-components of magnitude smaller than $\epsilon$. 
\end{enumerate}
 
Note that $\alpha_{E(1)}^2=\alpha_8^2+2g(c_1-g)>0$ implies that $c_1-g>0$.  Choose $c_1>g$ so that $0<\alpha_{E(1)}^2<\epsilon$.  The class $\alpha_N$ satisfies:\begin{enumerate}
\item $\alpha_N^2=\alpha^2-\alpha_{E(1)}^2>0$ and 
\item $\alpha_N\cdot E_i=\alpha\cdot E_i>0$,
\end{enumerate}
hence $\alpha_N\in \mathcal C^{F}_{N,K}\subset \mathcal C_{N}^{F_g}$.

In particular, choosing $\epsilon$ small enough, it can be achieved that 
\begin{enumerate}
\item $k_0\le 0$ with $|k_0|\le\frac{g}{4}$.
\item $k_i\ge 0$, $i\ge 1$, with $k_i<\frac{g}{4}$. 
\end{enumerate}
This implies that for $i\ne j$,
\begin{equation}\label{e:kest}
|k_0-k_i|<\frac{g}{2},\; |k_i-k_j|<\frac{g}{4} \mbox{  and  }|k_0+k_i-k_j|\le|k_0|+|k_i-k_j|<\frac{g}{2}.
\end{equation}
Combining $c_1-g>0$ with the estimates \eqref{e:kest} shows that each entry
of the standard form in \eqref{e:bc} is positive.  

Reflection on $E_i^K-E_j^K$ results in the two terms switching places in the standard form, this is also an automorphism of the fiber sum manifold $M$ by Theorem \ref{t:2} and thus may be applied to re-order the terms $b_1,...,b_8$.  This only changes $k_1,...,k_7$ in the split form, leaving both $k_0$ and the $(E,F)$-part unchanged.   Use such reflections on the standard form to sort the terms $b_i$ such that  $b_1\ge b_2\ge...\ge b_8$, determine the corresponding split form and rename the $k_i$ correspondingly.  Thus we may assume the standard form again has the structure of \ref{e:bc} with the central 8 terms ordered by size.  

To ensure that the class is reduced, it remains to show the estimate on the leading term $a$.  However, note that $b_9=c_1-g$ has not been sorted in the above re-ordering.  Thus it is necessary to consider 2 cases, in each calculating $T=a-b_i-b_j-b_k$:\begin{enumerate}
\item $\{b_1,b_2,b_3\}$: $T=-2k_0+k_3>0$ by the choice of signs on $k_0$ and $k_i$.
\item $\{b_1,b_2,b_9\}$: $T=-k_0-k_1+k_3+g$, thus the previous estimates show that $T\ge 0$.

\end{enumerate}

Thus by Lemma \ref{l:e1setup}.2 and Theorem \ref{t:be1},  $\alpha_{E(1)}\in\mathcal C^F_{E(1),K}\subset \mathcal C_{E(1)}^{F_g}$.  As it was shown above that $\alpha_N\in \mathcal C_{N}^{F_g}$, $\alpha$ is sum balanced.

Thus any class $\alpha_0\in\mathcal C_{M,K}^{F}$ is equivalent to a sum balanced class $\alpha\in\mathcal C_{M,K}^{F}$.   If $\alpha_{bal}=0$, then \cite{Go2} shows that  $\alpha\in \mathcal C_{M}^{F_g}$.  Otherwise, apply Theorem \ref{t:sym} (noting Lemma \ref{l:pfc} and the pfc condition on $N$) to the class $\alpha$.  It follows that $\alpha\in \mathcal C_{M}^{F_g}$ and hence also $\alpha_0$.

{\bf II:} Assume now that $N=E(1)$, this means that $M=E(2)$.  Let $\alpha_0\in\mathcal P^F_{E(2)}$.   Note that Lemma \ref{l:pfc} ensures the pfc-condition of Theorem \ref{t:sym} is satisfied.

 The argument for the $E(1)$ side in the previous case must now be applied on both sides.  This argument had two parts:  minimizing the size of the $E_8$ coefficients and ensuring that $c_1-g>0$.  Theorem \ref{t:e2} arranges for these conditions to hold in each $E(1)$ or $\alpha_0\in\mathcal P^F_M$ is equivalent to $w\tilde F+ g\tilde W$.  Note that in both cases the fiber class may have changed.
 
 In the first case, proceed as in I, and by choosing $\epsilon$ small enough, a splitting 
\[
\alpha_{E(1),1}+\alpha_{E(1),2}+\alpha_{bal}
\]
can be achieved with each $\alpha_{E(1),*}\in\mathcal C^F_{E(1),K}$.  Note that condition 4 of Theorem \ref{t:e2} ensures that $c> 2g$ and thus a splitting of $c=c_1+c_2$ with each $c_i>g$ can be achieved.

If $\alpha_0$ is equivalent to $w\tilde F+g\tilde W$ with $\frac{w}{g}\in \mathbb Q$, then choose $0<\delta<<1$ such that $\frac{w-\delta}{g}\not\in \mathbb Q$.  Then the previous case shows that this class lies in $\mathcal C_{E(2)}^{\tilde F_g}$.  Now apply inflation along the fiber to regain the class $w\tilde F+g\tilde W$.

As in Case I., this ensures that $\alpha_0\in \mathcal C_{E(2)}^{F_g}$.  Together with \eqref{e:incl},  $\mathcal C_{E(2)}^{F_g}=\mathcal P^F_{E(2)}$.

{\bf III:} Let  $\alpha_0=\alpha_{E(2)}-\sum_{i=1}^le_iE_i\in \mathcal C_{E(2)\#l\overline{\mathbb CP^2},K}^{F}$.

Apply Theorem \ref{t:e2} to the class $\alpha_{E(2)}$ and assume that Theorem \ref{t:e2}.4a holds for the equivalent class $\tilde\alpha_0$.  Then Theorem \ref{t:s2t2} can be applied and it follows that $\alpha_0\in\mathcal C^{F_g}_{E(2)\#l\overline{\mathbb CP^2}}$.

Further, the same argument can be applied to the case with $\frac{w}{g}\in \mathbb Q$.

%

\end{proof}

\subsection{$E(n,g)$: Relative Symplectic Cones}

This result can be extended to determine the relative symplectic cone for $E(n)$.

\begin{theorem}\label{t:en} Let $M=E(n)\# l\overline{\mathbb CP^2}$, $n>1$, be an elliptic surface and $F_g$ an oriented generic fiber.  Fix a symplectic canonical class $K\in\mathcal K_F$ on $M$ and let $\mathcal E_K=\{E_1,...,E_l\}$.  Then 
\[
\mathcal C^{F}_{M,K}\subset \mathcal C_M^{F_g}.
\]
In particular, 
\[
\mathcal C_M^{F_g}=\bigcup_{K\in\mathcal K_F}\mathcal C_{M,K}^F=\{\alpha\in\mathcal P^F_M\;|\; \alpha\cdot E_i\ne 0\;\forall i=1,..,l\}
\]
and if $l=0$, then $\mathcal C_M^{F_g}=\mathcal P_M^F$.  Moreover, every class in $\mathcal C_M^{F_g}$ is pfc at $F_g$.
\end{theorem}

\begin{proof}Consider first $M=E(n)$, $n\ge 2$:

Note the following useful fact:  If $M=X\#_{F_g}Y$ with $\mathcal C_{*}^{F_g}=\mathcal P_{*}^F$ for each summand, then a balanced class on $M$ is also sum balanced with respect to $(X,Y)$. 

Theorem \ref{t:e1} proves
\[
\mathcal C_{E(2)}^{F_g}= \mathcal P^{F}_{E(2)}.
\] 
Moreover, the construction of the symplectic form in the proof of Theorem \ref{t:sym} ensures that the sum form, which is compatible with the fibration structure on $E(2)$, is modified only in an arbitrarily small tubular neighborhood of the rim torus pair.  For the fibration $E(2)\rightarrow S^2$, this means there is a configuration of curves $C\subset S^2$ consisting of two parallel circles, arbitrarily close to each other, and an arc transversally meeting both circles on $S^2$, above which the symplectic form may not be compatible with the fibration. On the two disks defined by this configuration,  every form obtained by a sum balanced class is fibration compatible and each disk contains 5 pairs of singular fibers with matching vanishing cycles.  Hence any class $\alpha\in \mathcal C_{E(2)}^{F_g}$, where $F_g$ is not a fiber above $C$, is pfc at $F_g$.

If $F_g$ is above the configuration $C$, then the configuration $C$ can be chosen differently.  For the circle and line segment producing the torus in the class $\mathcal R+S$, choose a different path for the sphere and a different push off of the rim torus $R$.  For the rim torus $R$, change the size of the neighborhood removed from around $\tilde F_g$ to produce $E(2)=E(1)\#_{\tilde F_g}E(1)$.

As each sum-balanced class is obtained via some automorphism of the elliptic surface, this is in fact true for every class in $\mathcal C_M^{F_g}$.  Hence, for any smooth fiber $F_g$, every class $\alpha\in \mathcal C_{E(2)}^{F_g}$ is pfc at $F_g$.  

This argument is essentially local and applies also to the case $E(2)\#l\overline{\mathbb CP^2}$, however now fibers containing an exceptional sphere are excluded.  Let $M=E(2)\# l\overline{\mathbb CP^2}$ and let  $p_i\in S^2$ be the base points of the fibers in $E(2)$ which contain an exceptional curve.  Denote by $M'$ the restriction of $M$ over $S^2/\{p_1,...,p_l\}$.  The construction of the symplectic form representing a given class in Theorem \ref{t:sym} show that any sum balanced class in $\mathcal C_M^{F_g}$ is represented by a symplectic form compatible with the fibration on $M'$ except over $n-1$ disjoint curves in a $C$-type configuration.  Again, as each sum-balanced class is obtained via some automorphism of the elliptic surface, this is in fact true for every class in $\mathcal C_M^{F_g}$. 

Moreover, this allows for an inductive determination of the relative symplectic cone for $E(n)$, the pfc arguments given for $E(2)$ are similar in the $E(n)$ case:
\[
E(2m)=E(2(m-1))\#_{\tilde F_g}E(2)\stackrel{Thm.\;\ref{t:sym}}{\Rightarrow}\mathcal C_{E(2m)}^{F_g}= \mathcal P^{F}_{E(2m)}
\]
and
\[
E(2m+1)=E(2m)\#_{\tilde F_g}E(1)\stackrel{Thm. \;\ref{t:e1}}{\Rightarrow}\mathcal C_{E(2m+1)}^{F_g}= \mathcal P^{F}_{E(2m+1)},\;\;(m>0).
\]

Assume now that $M=E(n)\#l\overline{\mathbb CP^2}$ with $n>1$ and $l>0$.  Note that the exclusion of $E(1)\#l\overline{\mathbb CP^2}$ means that $M$ has exactly the $l$ exceptional spheres in $\mathcal E_K$.  This allows the useful fact from the minimal case to continue to hold:  If $M=X\#_{F_g}(Y\#l\overline{\mathbb CP^2})$ with $\mathcal C_{X}^{F_g}=\mathcal P_{X}^F$ and
\[
\left\{\left.\alpha\in \mathcal P_{Y}^{F}\;\right|\; \alpha\cdot E_i>0\; \forall i=1,..,l\right\}\subset \mathcal C^{F_g}_{Y},
\]
then a balanced class on $M$ with $\alpha_0\cdot E_i>0$ is also sum balanced with respect to $(X,Y\#l\overline{\mathbb CP^2})$. 

The case $E(2) \#l\overline{\mathbb CP^2}$ is covered by Theorem \ref{t:e1}.  As before, the pfc argument applies to every class.

If $n=2m\ge 4$, then 
\[
E(2m)\#l\overline{\mathbb CP^2}=(E(2)\#l\overline{\mathbb CP^2})\#_{\tilde F_g}E(2(m-1))
\]
and the results for $E(2(m-1))$, $2(m-1)\ge 2$, together with Theorem \ref{t:e1} imply the result by Theorem \ref{t:exc}.  If $n=2m+1\ge 3$, then decompose as
\[
E(2m+1)\#l\overline{\mathbb CP^2}=(E(2)\#l\overline{\mathbb CP^2})\#_{\tilde F_g}E(2m-1).
\]
As before, the results in the minimal case together with Theorem \ref{t:e1} and Theorem \ref{t:exc} imply the result. Combining all these cases, it follows that
\[
\mathcal C^{F_g}_{E(n)\#l\overline{\mathbb CP^2},K}\subset \mathcal C_{E(n)\#l\overline{\mathbb CP^2}}^{F_g}.
\]

\end{proof}

\begin{theorem}\label{t:full} Let $M_I$ be an elliptic surface without multiple fibers and $F_g$ an oriented generic fiber.  Let $M=M_I\# l\overline{\mathbb CP^2}$.  Assume $M$ is not diffeomorphic to $ E(1)\# l\overline{\mathbb CP^2}$, $l\ge 0$.  Fix a symplectic canonical class $K\in\mathcal K_F$ on $M$ and let $\mathcal E_K=\{E_1,...,E_l\}$.  Then 
\[
\mathcal C^{F}_{M,K}\subset \mathcal C_M^{F_g}.
\]
In particular, 
\[
\mathcal C_M^{F_g}=\bigcup_{K\in\mathcal K_F}\mathcal C_{M,K}^F=\{\alpha\in\mathcal P^F_M\;|\; \alpha\cdot E_i\ne 0\;\forall i=1,..,l\}
\]
and if $l=0$, then $\mathcal C_M^{F_g}=\mathcal P_M^F$.
\end{theorem}

\begin{proof}Consider the manifolds $E(n,g)\#l\overline{\mathbb CP^2}$.  If $g=0$, then this result is just Theorem \ref{t:en}.  Hence assume that $g>0$.

{$\bf E(n,g)$:} Consider  the decomposition
\[
E(n,g)=E(n)\#_{\tilde F_g}(T^2\times \Sigma_g).
\]

The case $n=1$ is the content of Theorem \ref{t:e1}.  For $n\ge 2$, the claim then follows from Cor. \ref{c:chi0}, Theorem \ref{t:en} and applying Theorem \ref{t:38}:
\[
\mathcal C^{F_g}_{E(n,g)}=\mathcal P^{F_g}_{E(n,g)}.
\]
Moreover, Cor. \ref{c:chi0}, Lemma \ref{l:pfc} and Theorem \ref{t:en} imply that every class $\alpha\in \mathcal C^{F_g}_{E(n,g)}$ is pfc at $F_g$.

{\bf $\bf E(n,g)\#l\overline{\mathbb CP^2}$:} If $n\ge 2$, consider the decomposition
\[
E(n,g)\#l\overline{\mathbb CP^2}=(E(n)\# l\overline{\mathbb CP^2})\#_{\tilde F_g}(T^2\times \Sigma_g).
\]
Then Cor. \ref{c:chi0}, Theorem \ref{t:en} and Theorem \ref{t:exc} imply the result, again no rim-tori are involved.  As before, every class is pfc at $F_g$.
For the case with $n=1$, decompose as
\[
E(1,g)\#l\overline{\mathbb CP^2}=E(1)\#_{\tilde F_g}[(T^2\times \Sigma_g)\#l\overline{\mathbb CP^2}]
\]
and use Cor. \ref{c:chi0}, Theorem \ref{t:be1} and Theorem \ref{t:e1} to obtain the result.

Combining these results, it follows that for the cases considered 
\[
\mathcal C^{F_g}_{E(n,g)\#_{\tilde F_g}l\overline{\mathbb CP^2},K}\subset \mathcal C_{E(n,g)\#_{\tilde F_g}l\overline{\mathbb CP^2}}^{F_g}.
\]

\end{proof}

\subsection{$E(n,g,p_1,..,p_k)$: Relative Symplectic Cones}  To deal with multiple fibers, we use the rational blow-down introduced in \cite{FS}.  In the following we briefly describe this procedure in the setting that we need, its connection to logarithmic transforms and Symington's construction of symplectic forms on the rational blowdown.

The rational blow-down can be performed as follows (Section 8.5, \cite{GS}):  Consider a fishtail fiber in $M$.   Blow-up the nodal point to produce a $-4$-sphere which is intersected by the exceptional sphere in two points.  Blow-up one of the two intersection points to produce a $-5$-sphere intersecting the new exceptional sphere in a single point.  Blow up $r-2$ times to obtain a collection of spheres $\{S_{-r-2}, S_{-2}, ..., S_{-2}\}$, each $S_j$ of self-intersection $j$ and such that each sphere in the list intersects only its neighbors transversally in a single point.  This linear tree can be plumbed to produce a  manifold $C_r$ with boundary the Lens space $L(r^2,r-1)$.  This Lens space bounds a rational homology ball $B_r$, hence removing $C_r$ and gluing in $B_r$ produces a new manifold $M_r$, the rational blow-down of $M$.

A key reason for introducing this construction is the following result:

\begin{theorem}(Thm. 3.1, \cite{FS}; Thm. 8.5.9, \cite{GS})\label{t:ratbl}  The manifold $M_r$ obtained as described above is diffeomorphic to the logarithmic transform $M(r)$ of multiplicity $r$ along a smooth fiber torus lying in the neighborhood of the fishtail fiber.
\end{theorem}

In particular, this diffeomorphism is identity outside the neighborhood of the fishtail fiber.  Hence the fibration is only affected in a neighborhood of a fishtail fiber, any submanifolds not intersecting this fiber can be assumed to be preserved.

Thus, to understand symplectic forms on $E(n,g,p_1,..,p_k)$ is equivalent to understanding symplectic forms on $E(n,g)_{p_1,..,p_k}$.

Thus far, this operation has been performed in the smooth category.  Symington \cite{Sym} proved that this can be extended to the symplectic setting. 

\begin{theorem}\cite{Sym}\label{t:symin} If $M$ is symplectic and each sphere $S_i\subset M$ of a configuration $C_r$ is a symplectic submanifold, then the rational blowdown $M_r$ is also symplectic.
\end{theorem}

The key to the proof is the construction of a number of model neighborhoods in the symplectic setting.  In particular, the symplectic analog of $B_r$ is constructed out of the complement of two spheres in a rational ruled surface $B_r'$ and an additional collar piece $A_r$.  $B_r'$ is obtained by removing from the rational ruled surface $F_{r-1}$ two spheres in the classes $[\Sigma_{r-1}]+F$ and $[\Sigma_{-r+1}]$ of area $\sigma_{r+1}$ and $\sigma_{-r+1}$ respectively.  In order to ensure that all gluings can occur symplectically, these two areas need to satisfy
\begin{equation}\label{e:est}
(r-1)\sigma_{r+1}+\sigma_{-r+1}<(r-1)a_1+ \sum_{i=2}^{r-1}\tau_ia_i
\end{equation}
where $a_i$ are the areas of each of the spheres $S_i$ and $\tau_i$ are positive constants defined in \cite{Sym}.  This strict inequality explains why the volume of $M_r$ is determined by the volume of $M$ and the areas of the $S_i$, in particular why 
\[
vol(M_r)>vol(M)+vol(B'_r).
\]

Combining these results now allows us to prove the following Theorem.

\begin{theorem}\label{t:mult}Let $M$ be an elliptic surface with $\chi(M)>0$ and $F_g$ an oriented generic fiber.  Assume $M$ is not diffeomorphic to $ E(1)\# l\overline{\mathbb CP^2}$, $l\ge 0$.  Fix a symplectic canonical class $K\in\mathcal K_F$ on $M(r)$ and let $\mathcal E_K=\{E_1,...,E_l\}$.  Then 
\[
\mathcal C^{F}_{M(r),K}\subset \mathcal C_{M(r)}^{F_g}.
\]
In particular, 
\[
\mathcal C_{M(r)}^{F_g}=\bigcup_{K\in\mathcal K_F}\mathcal C_{M(r),K}^F=\{\alpha\in\mathcal P^F_{M(r)}\;|\; \alpha\cdot E_i\ne 0\;\forall i=1,..,l\}
\]
and if $l=0$, then $\mathcal C_{M(r)}^{F_g}=\mathcal P_{M(r)}^F$.
\end{theorem}

\begin{proof}
Every elliptic surface $M$ with $\chi(M)>0$ contains a fishtail fiber.  As we exclude $E(1)$,  the exceptional spheres in $M$ can be assumed to be away from the fishtail fiber  Note that any exceptional curves used to produce $C_r$ do meet this fishtail fiber, but they disappear in the rational blowdown.

Now apply Theorem \ref{t:ratbl} to switch from $M(r)$ to $M_r$.  The diffeomorphism only affects a neighborhood of the fishtail fiber, it leaves the exceptional spheres $E_1,..,E_l$ unchanged, as well as the $E_8$ and $H$ components of the intersection form except possibly for the pair $(F,\Gamma)$.  However, the fiber class exists outside of the neighborhood, so it also stays preserved, and with it the class $\Gamma$, as the intersection form is preserved by the diffeomorphism.  Thus, to prove the Theorem, the same cones must be shown to be symplectic in $M_r$.

Let $\alpha_r\in \mathcal C^{F}_{M_r,K}$.  Then $\alpha_r=cF+g\Gamma+B$ with $g>0$ and $B\cdot E_i>0$.  The class $\alpha(\beta)=\beta F+g\Gamma+B\in H^2(M,\mathbb R)$ therefore also satisfies $\alpha(\beta)\cdot E_i>0$.  If $\beta$ is chosen so that $\alpha(\beta)^2>0$, then Theorem \ref{t:full} ensures that $\alpha(\beta)\in \mathcal C^{F_g}_M$.  In particular, there exists an interval $I=(\tau,+\infty)\subset \mathbb R$, such that for all $\beta\in I$, $\alpha(\beta)\in \mathcal C^{F_g}_M$.

If $r=2$, then $C_2$ is a single sphere of self-intersection $-4$ and the rational blow-down is a symplectic sum with $\mathbb CP^2$ along a sphere in the class $2H$.  As described by Gompf \cite{Go2}, the symplectic form on $M_2$ has volume given by the sum of the volumes of the forms on $M$ and $\mathbb CP^2$.  Thus choose $\beta=c-\epsilon$ and the symplectic form on $\mathbb CP^2$ in the class $2\epsilon gH$.  Then the symplectic sum will produce a form in the class of $\alpha_r$, hence  $\alpha_r\in  \mathcal C_{M(r)}^{F_g}$.

Now let $r>2$.  For each $\beta\in I$, Symington's construction builds a symplectic form on $M_r$.  The volume of this form is determined by the area of the fishtail fiber and the size of the blow-ups used to create $C_r$.  More precisely, choose $\sigma_{r+1}=a_1=g$, small blow-ups $a_i$ ($i\ge 2$) and $\sigma_{-r+1}$ small to ensure that Eq. \ref{e:est} holds.  Then the volume of $B_r'$ can be chosen to be small.  Moreover, the symplectic structure $\alpha(\beta)_r$ arising from $\alpha(\beta)$ by Symington's construction will then have volume only slightly larger than the volume of $\alpha(\beta)$.  Therefore, choosing the terms appropriately, it can be achieved that 
\begin{enumerate}
\item $\alpha(\beta)_r=\alpha_r$ and
\item $\alpha(\beta)^2>0$.
\end{enumerate}
Hence $\alpha_r\in  \mathcal C_{M(r)}^{F_g}$.

\end{proof}

\section{Symplectic Cones for Elliptic Surfaces} \label{s:cone}

The full symplectic cone for elliptic surfaces with $b^+=1$ is known to be (Theorem 4, \cite{TJLL})
\[
\mathcal C_M=\{\alpha\in\mathcal P_M\;|\;\alpha\cdot E\ne 0\;\forall E\in\mathcal E\}.
\]
If $M=E(0)$, then the canonical class is unique, hence this is the union of the relative symplectic cones.

If $M=E(1)\#l\overline{\mathbb CP^2}$, there are many exceptional curves and many symplectic canonical classes, hence the symplectic cone is much larger than the relative symplectic cones for a fixed fiber.  

For elliptic surfaces with $b^+\ge 2$, the relative symplectic cones determined in the previous section suffice to determine the full symplectic cone of the underlying smooth manifold.  In the following these cones are described with a primary emphasis on elliptic surfaces with $\chi(M)>0$. 

\subsection{$\bf b^+\ge 2$ $\bf \kappa(M)=0$ and $\bf \chi(M)>0$:}

If $M=E(2)$ is minimal, then $\mathcal C_M=\mathcal P_M$ \cite{TJL1}.  The non-minimal case is given by the following lemma.

\begin{lemma} \label{l:k3} 
\[
\mathcal C_{E(2)\#l\overline{CP^2}}=\{\alpha\in\mathcal P_{E(2)\#l\overline{CP^2}}\;|\;\alpha\cdot E_i\ne 0\;\forall i\in\{1,..,l\}\}
\]
\end{lemma}

\begin{proof} Clearly the inclusion $\subset$ holds.

Let $M=E(2)\#l\overline{\mathbb CP^2}$.  Let $\alpha\in\mathcal P_{E(2)\#l\overline{\mathbb CP^2}}$, write as 
\[
\alpha=(\alpha_{8,1},\alpha_{8,2},\underbrace{c,g}_{\langle F,\Gamma\rangle =H},\underbrace{a_1,b_1}_{=H},\underbrace{a_2,b_2}_{=H},e_1,...,e_l).
\] 
If $\alpha\cdot F=g\ne 0$, then Theorem \ref{t:e1} shows that $\alpha\in\mathcal C_{M}$ for some choice of orientation on $F_g$. Suppose that $\alpha\cdot F=0$, then the $(F,\Gamma)$-terms do not contribute to the volume of $\alpha$.  As $\alpha^2>0$ and $\alpha_{8,1}^2+\alpha_{8,2}^2\le 0$, it follows that at least one of $(a_i,b_i)$ must have positive square, suppose $(a_1,b_1)$.  Then Lemma \ref{l:e2} maps the fiber $F$ to $\tilde F=F+T_1$ and the class $\alpha$ to 
\[
\tilde\alpha=(\alpha_{8,1},\alpha_{8,2},c+b_1,b_1,a_1-c,b_1,a_2,b_2,e_1,...,e_l)
\]
(written with respect to the new fiber $\tilde F$) such that $\alpha$ and $\tilde \alpha$ are in the same orbit under $O'$.  Now $\tilde\alpha\cdot \tilde F=b_1\ne 0$ and hence $\tilde\alpha$ lies in the symplectic cone.
\end{proof}

%

\subsection{$\bf b^+\ge 2$, $\bf \kappa(M)=1$ and $\bf \chi(M)>0$:}  If $M$ is (relatively) minimal with $b^+>1$, then there is a unique symplectic canonical class, up to a sign, see \cite{Br} and \cite{FM}.  Hence 
\[
\mathcal C_M=\mathcal P_M^F\cup \mathcal P_M^{-F}.
\]
In the non-minimal case, $M$ has exactly $l$ exceptional spheres.  These exceptional curves provide further restrictions, however Theorems \ref{t:full} and \ref{t:mult} imply
\[
\mathcal C_M=\mathcal C_M^{F_g}\cup \mathcal C_M^{-F_g}=\{\alpha\in\mathcal P_{M}\;|\;\alpha\cdot F\ne 0,\;\;\alpha\cdot E_i\ne 0\;\forall i\in\{1,..,l\}\}
\]
%

\end{document}